\documentclass[a4paper,11pt]{amsart}
\usepackage{amsfonts}
\usepackage{amsthm}
\usepackage{amssymb}
\usepackage{amsmath}
\usepackage[T1]{fontenc}
\usepackage[utf8]{inputenc}
\usepackage{mathrsfs}
\usepackage{verbatim}
\usepackage{tikz}
\usepackage{hyperref}
\theoremstyle{plain} 
\newtheorem{Lemma}{Lemma}[section] \newtheorem{Thm}[Lemma]{Theorem} \newtheorem{Prop}[Lemma]{Proposition}\newtheorem{Cor}[Lemma]{Corollary} \newtheorem{Conj}[Lemma]{Conjecture}
\theoremstyle{definition} \newtheorem{Defn}[Lemma]{Definition} \newtheorem{Ex}[Lemma]{Example}
\theoremstyle{remark} \newtheorem{Rem}[Lemma]{Remark} 
\numberwithin{equation}{section}

\newcommand{\ZZ}{\mathbb{Z}}
\newcommand{\PP}{\mathbb{P}}
\newcommand{\CC}{\mathbb{C}}
\newcommand{\QQ}{\mathbb{Q}}
\newcommand{\RR}{\mathbb{R}}
\newcommand{\OO}{\mathcal{O}}

\newcommand{\eps}{\varepsilon}
\newcommand{\dd}{\partial}
\newcommand{\inj}{\hookrightarrow}

\newcommand{\sse}{\subset}

\begin{document}
\author{Ketil Tveiten}
\title{Period Integrals and Mutation}
\address{Ketil Tveiten\\Department of Mathematics\\Stockholm University\\ 106 91 Stockholm.}
\email{ktveiten@math.su.se}

\maketitle

\begin{abstract}
Let $f$ be a Laurent polynomial in two variables, whose Newton polygon strictly contains the origin and whose vertices are primitive lattice points, and let $L_f$ be the minimal-order differential operator that annihilates the period integral of $f$. We prove several results about $f$ and $L_f$ in terms of the Newton polygon of $f$ and the combinatorial operation of \emph{mutation}, in particular we give an in principle complete description of the monodromy of $L_f$ around the origin. Special attention is given to the class of \emph{maximally mutable} Laurent polynomials, which has applications to the conjectured classification of Fano manifolds via mirror symmetry.
\end{abstract}

\section{Introduction}\label{sec:intro}

Let $N\simeq\ZZ^d$ be a lattice, and $\CC[N]$ the ring of Laurent polynomials in $d$ variables. Let $P\sse N$ be a lattice polytope, and let $f(a,x)=\sum_{m\in P}a_mx^m$ be a generic Laurent polynomial with Newton polytope $Newt(f)=P$. For such an $f$, and $C$ an $d$-cycle in $H_d(\{x\in (\CC^*)^d|f(x)\neq 0\},\CC)$, consider the integral
\[
\phi_f(a)=\int_{C}\frac{1}{f(a,x)}\frac{dx_1}{x_1}\cdots \frac{dx_n}{x_n}.
\]
As a function of the coefficients $a_m$ of $f$, this integral satisfies a system of differential equations of the GKZ type, as follows: Let $P'$ be the image of $P$ under the embedding $N\inj \ZZ\times N$ ``at height 1'' given by $m\mapsto (1,m)$. Then $\phi(f)$ is a solution to the GKZ system $H_\gamma(P')$, where $\gamma=(0,\ldots,0,-1)$ (see \cite[5.4.2]{SST}). 

 An interesting variant of this is the \emph{classical period integral of $f$},
 \[
 \pi_f(t)=(\frac{1}{2\pi i})^n\int_{|x_1|=\cdots =|x_n|=\eps}\frac{1}{1-tf(a,x)}\frac{dx_1}{x_1}\cdots \frac{dx_n}{x_n}
 \]
which is a (possibly multivalued) holomorphic function of $t$ in a punctured disk around $t=0$. The GKZ system annihilating $\phi_f$ specializes to a Picard-Fuchs operator $L_f=\sum p_i(t)\nabla^i\in \CC[t,\nabla_t]$ (here $\nabla_t=t\dd_t$). This operator $L_f$ plays an important role in the conjectured classification of Fano manifolds via mirror symmetry, for a certain class of Laurent polynomials $f$ called \emph{maximally mutable} \cite{CCGGK,overarching}; understanding this class of Laurent polynomials and their associated operators in the two-dimensional case is the primary motivation for this paper, though we will present everything in as general terms as possible.

We will study the local behaviour of $L_f$ around the origin by using tools from toric geometry, and in particular the relation between properties of $L_f$ and the combinatorial data of the Newton polytope $P$, in particular through the operation of \emph{mutation} \cite{ACGK}; in addition, we try to extract some information about the global behaviour of $L_f$ from the combinatorics of $P$. The former problem we can in principle solve completely (although we only do it explicitly for simple cases), for the latter the best we can do is make plausible conjectures backed up by empirical data. We also prove several results about the relationship between Laurent polynomials and their Newton polygons.


%
%

\section{Preliminaries}\label{sec:prelims}

We can of course explicitly compute $L_f$, by several different methods, the most efficient of which are recently developed by Lairez in \cite{Lairez}.
The computation is expensive, and for larger examples in practice undoable without fixing values for the coefficients. We hope to work around this problem by applying the Riemann-Hilbert correspondence; the operator $L_f$ is equivalent to the monodromy data of the solution sheaf $Sol(L_f,\OO)$, which is a local system away from the singular points of $L_f$. This lets us work with general $f$, unfortunately we pay the price of only being able to talk about the single singular point $t=0$. The problem of finding the remaining singular points and their local monodromy data is again a prohibitively expensive computational problem, though we will conjecture ways to extract some information in the final section.

As we will discuss the operation of \emph{mutation} of polygons and Laurent polynomials, we must restrict to the class of \emph{Fano polygons}, where this operation is well-behaved.

\begin{Defn}\label{Defn:Fano-polygon}
Let $N$ be a two-dimensional lattice and let $P\sse N\otimes \RR$ be a convex lattice polygon such that
\begin{enumerate}
\item $\dim P = 2$;
\item $0\in int(P)$, that is, the origin is a strict interior point of $P$; and
\item the vertices of $P$ are primitive lattice points.
\end{enumerate}
Such a polygon is called a \emph{Fano polygon} (see \cite{KN12}).
\end{Defn}

In the remainder, all polygons are assumed to be Fano polygons, and all Laurent polynomials are such that $Newt(f)$ is Fano. We consider two polygons to be equal if they differ by an element of $GL(N)$, and similarly consider two Laurent polynomials to be equal if they are related by an automorphism of $\CC[N]$ induced by an element of $GL(N)$. The one-dimensional faces of a polygon are called \emph{edges} and the zero-dimensional faces are called \emph{vertices}.

Let $P\sse N$ be a Fano polygon, let $M=Hom(N,\ZZ)$, and let $Y_P$ be the toric del Pezzo surface defined by the normal fan of $P$ in $M$. Recall that a variety $Y$ is \emph{Fano} if the anticanonical divisor $-K_Y$ is very ample (two-dimensional Fano varieties are usually called \emph{del Pezzo surfaces} for historical reasons). The rays $u_i$ generating the normal fan of $P$ are the inward normals to the edges $E_i$ of $P$, so for each edge $E_i$ of $P$, let $D_{E_i}=D_i$ denote the corresponding divisor on $Y_P$. There is a distinguished divisor on $Y_P$,
\[
D_P=\sum_ih_iD_i,
\]
here $h_i=-\langle u_i|E_i\rangle$ is the lattice height of the edge $E_i$. $D_P$ is very ample, and its global sections $\Gamma(Y_P,D_P)$ can be identified with the set of Laurent polynomials with Newton polygon $Newt(f)=P$ (see \cite[4.3.3/4.3.7]{Cox-Little-Schenk}). A section $f$ (or generally a linear system $\delta\sse \Gamma(Y_P,D_P)$ of sections) determines a rational map 
\[
\tau:=\frac1f:Y_P\dashrightarrow\PP^1.
\]
Let $\Gamma_\tau=\{(y,\tau(y))\in Y_P\times \PP^1\}$ be the graph of $\tau$. Via the embedding $Y_P\sse \Gamma_\tau$ there is an induced rational map $\Gamma_\tau\dashrightarrow\PP^1$, and if we resolve all singularities of $\Gamma_\tau$ (and generally any base points of $\delta$) we get a smooth surface $\widetilde{Y_P}$ such that the induced map $\widetilde{\tau}:\widetilde{Y_P}\to\PP^1$ is a morphism. Let $D$ be the pullback of $D_P$ to $\widetilde{Y_P}$. The fiber $X_0=\widetilde{\tau}^{-1}(0)$ is equal to the support of $D$, and the fiber $X_t$ over a general point $t\in \PP^1$ is smooth. 

Now by general $D$-module theory the solution sheaf $Sol(L_f)$ on $\PP^1$ is isomorphic to the constructible sheaf with fiber $H_1(X_t,\CC)$ at $t\in \PP^1$; here $H_1$ denotes \emph{homology with closed support} (also known as \emph{Borel-Moore homology}), and not the usual singular homology. Indeed, $L_f$ is a $D$-module theoretic direct image of the GKZ module, and from \cite{Borel} (VII.9.6, VIII.13.4 and VIII.14.5.1) we have that the solution complex $Sol(L_f)^\bullet$ is the exceptional direct image of the cohomology complex of that module, which works out to give the described constructible sheaf. In other words, we wish to find the monodromy of $H_1(X_t,\CC)$ around $t=0$.

%
%

\section{Mutation}\label{sec:mutation}

We introduce some notation and terminology which will remain in force for the remainder of the paper.

Let $P\sse N$ be as before, with vertices $p_i$ and edges $E_i$ with inward normal vectors $u_i\in M$, we number these so that $E_i$ is the edge between $p_i$ and $p_{i+1}$. The \emph{lattice height} of an edge $E_i$ is $-\langle u_i|E_i\rangle$, and the \emph{lattice width} of $E_i$ is $\langle u_{i-1}|p_{i}-p_{i+1}\rangle = \langle u_{i+1}|p_{i+1}-p_i\rangle$; the lattice width is equal to the number of lattice points on $E_i$ minus one. The following definition is due to Akhtar and Kasprzyk (see \cite{Al-Mo-singularity-content}).

\begin{Defn}\label{Defn:R/T-cones}
  Let $C\sse N$ be a primitive lattice cone of lattice height $h$ and lattice width $w$. If $h=w$, we say that $C$ is a \emph{primitive $T$-cone}. If $w$ is a positive multiple of $h$, we say that $C$ is a \emph{$T$-cone}. If $w$ is strictly less than $h$, we say that $C$ is an \emph{$R$-cone}.
\end{Defn}

Let $E$ be an edge of $P$ of height $h$ and width $w=hk+r$. The cone over $E$ from the origin may be subdivided into $k$ primitive $T$-cones and an $R$-cone of width $r$; we say that these cones are \emph{on the edge $E$}. There are $k+1$ ways to do this, so saying e.g. ``the $R$-cone on the edge $E$'' is strictly speaking not well-defined, but this observation is not relevant anywhere in what follows (i.e. any subdivision gives the same result), so we permit ourselves to abuse language in this way. The most important quantity here is the number of internal lattice points of $P$ lying in $R$-cones (or later, $f$-rigid cones, see \ref{Defn:f-mutation}), which does not depend on the subdivision, by \cite[2.3]{Al-Mo-singularity-content}. Whenever we say e.g. ``internal points of an $R$-cone'', it should always be understood to mean lattice points in $P$.

\begin{Defn}\label{Defn:type-of-cones}
  Let $C\sse N$ be a primitive lattice cone, with primitive spanning vectors $u$ and $v$. If $\{u,v\}$ is a lattice basis for $N$, we say that $C$ is \emph{smooth}. If $C$ is not smooth, there is a point $p\in N$ such that $p = \frac1r u+\frac{a}{r} v$, and $\{u,p\}$ and $\{v,p\}$ are lattice bases for $N$; in this case we say that $C$ is \emph{of type $\frac1r(1,a)$}.
\end{Defn}

\begin{Rem}
  The type of cones parallels the classification of cyclic quotient singularities; a cone of type $\frac1r(1,a)$ defines a toric variety isomorphic to the cyclic quotient singularity $\CC^2/\mu_r$, where $\mu_r$ acts with weight $(1,a)$. See \cite{Al-Mo-singularity-content} for further details on $R$- and $T$-cones, and the corresponding singularities, called $R$- and $T$-singularities. The operation of \emph{mutation}, which we will now describe, (conjecturally, see \cite{overarching}) corresponds to $\QQ$-Gorenstein deformation of toric varieties; roughly speaking $T$-cones (resp. $T$-singularities) are mutable (resp. $\QQ$-Gorenstein smoothable), while $R$-cones and $R$-singularities are rigid under mutation/deformation.
\end{Rem}

\begin{Defn}[\cite{ACGK}]\label{Defn:P-mutation}
  Let $P\in N$ be a Fano polygon, and let $E$ be an edge of $P$, with inward normal vector $u$, lattice height $h$, and lattice width $w=kh+e$, where $k>0,e\ge 0$ are integers (in other words, $E$ supports $k$ $T$-cones and an $R$-cone of width $e$), let $h'$ be the minimal lattice height (with respect to $u$) of the points in $P$, and let $F\in u^\bot\sse N$ be a primitive vector. For any $r\in \ZZ$ let $P_r=\{p\in P|u(p)=r\}$ be the points of $P$ at height $r$ with respect to $u$ (in particular, $P_{h}=E$). Notice that we can for each $0<r\le h$ write $P_r$ as a Minkowski sum $kr\cdot F + Q_r$, where $Q_r$ is some (possibly empty) polygon. The \emph{mutation of $P$} with respect to the mutation data $u,F$ is the polygon $P'=mut_u(P)$ defined by
\[
P'_{r}=\begin{cases} (k-1)r F+Q_r & 0\le r\le h\\ P_r+rF & h'\le r< 0.\end{cases}
\]
Intuitively, we are removing slices $rF$ from each positive height $r>0$, and adding slices $r'F$ at each negative height $r'<0$. Equivalently, we are contracting a $T$-cone from $E$, and putting in a $T$-cone on the opposite side of $P$.

Any polygons $P$,$Q$ related by a chain of mutations are said to be \emph{mutation-equivalent}.
\end{Defn}

We observe that $R$-cones are rigid under mutation: they do not have sufficient lattice width to permit mutation.

\begin{Ex}\label{Ex:mutation}
  Let $P$ be the Fano polygon with vertices $(-2,1),(-1,2),(3,2),(3,-1)$ and $(-2,-1)$. It has one $R$-cone of type $\frac13(1,1)$ (shaded dark grey), and nine $T$-cones. We will perform a mutation with factor $F=\{(-1,0),(0,0)\}$ (indicated by an arrow) and height function $h((x,y))=-y$, which will contract away the lightly shaded $T$-cone, and add a new $T$-cone on the other side of the polygon.
\[
  \begin{tikzpicture}[auto]
\fill[black!30] (0,0) -- (-2,1) -- (-1,2) -- (0,0);
\fill[black!10] (0,0) -- (-1,2) -- (1,2) -- (0,0);
\draw (-2,1) -- (-1,2) -- (3,2) -- (3,-1) -- (-2,-1) -- (-2,1);
\node (00) at (0,0) [shape=circle,fill=white,draw]{};
\node (m12) at (-1,2) {$\bullet$};
\node (02) at (0,2) {$\bullet$};
\node (12) at (1,2) {$\bullet$};
\node (22) at (2,2) {$\bullet$};
\node (32) at (3,2) {$\bullet$};
\node (m21) at (-2,1) {$\bullet$};
\node (m11) at (-1,1) {$\cdot$};
\node (01) at (0,1) {$\cdot$};
\node (11) at (1,1) {$\cdot$};
\node (21) at (2,1) {$\cdot$};
\node (31) at (3,1) {$\bullet$};
\node (m20) at (-2,0) {$\bullet$};
\node (m10) at (-1,0) {$\cdot$};
\node (10) at (1,0) {$\cdot$};
\node (20) at (2,0) {$\cdot$};
\node (30) at (3,0) {$\bullet$};
\node (m2m1) at (-2,-1) {$\bullet$};
\node (m1m1) at (-1,-1) {$\bullet$};
\node (0m1) at (0,-1) {$\bullet$};
\node (1m1) at (1,-1) {$\bullet$};
\node (2m1) at (2,-1) {$\bullet$};
\node (3m1) at (3,-1) {$\bullet$};

\draw[dashed] (00) to (m12);
\draw[dashed] (00) to (12);
\draw[dashed] (00) to (32);
\draw[dashed] (00) to (3m1);
\draw[dashed] (00) to (m21);
\draw[dashed] (00) to (m2m1);
\draw[dashed] (00) to (m1m1);
\draw[dashed] (00) to (1m1);
\draw[dashed] (00) to (0m1);
\draw[dashed] (00) to (2m1);
\draw[dashed] (00) to (3m1);
\draw[->,very thick] (00) to node {} (m10);
\node (flabel) at (-1,0) {$F$};

  \end{tikzpicture}
\]
After the mutation, we have this picture; the lightly shaded $T$-cone has been contracted, a new $T$-cone has been added to the opposite side, and the $R$-cone and the $T$-cone beneath it have been skewed to fit.
\[
  \begin{tikzpicture}[auto]
\fill[black!30] (0,0) -- (-1,1) -- (1,2) -- (0,0);
\fill[black!10] (0,0) -- (-2,-1) -- (-3,-1) -- (0,0);
\draw (-1,1) -- (1,2) -- (3,2) -- (3,-1) -- (-3,-1)--cycle;
\node (00) at (0,0) [shape=circle,fill=white,draw]{};
\node (12) at (1,2) {$\bullet$};
\node (22) at (2,2) {$\bullet$};
\node (32) at (3,2) {$\bullet$};
\node (m11) at (-1,1) {$\bullet$};
\node (01) at (0,1) {$\cdot$};
\node (11) at (1,1) {$\cdot$};
\node (21) at (2,1) {$\cdot$};
\node (31) at (3,1) {$\bullet$};
\node (m20) at (-2,0) {$\bullet$};
\node (m10) at (-1,0) {$\cdot$};
\node (10) at (1,0) {$\cdot$};
\node (20) at (2,0) {$\cdot$};
\node (30) at (3,0) {$\bullet$};
\node (m3m1) at (-3,-1) {$\bullet$};
\node (m2m1) at (-2,-1) {$\bullet$};
\node (m1m1) at (-1,-1) {$\bullet$};
\node (0m1) at (0,-1) {$\bullet$};
\node (1m1) at (1,-1) {$\bullet$};
\node (2m1) at (2,-1) {$\bullet$};
\node (3m1) at (3,-1) {$\bullet$};

\draw[dashed] (00) to (m11);
\draw[dashed] (00) to (12);
\draw[dashed] (00) to (32);
\draw[dashed] (00) to (2m1);
\draw[dashed] (00) to (1m1);
\draw[dashed] (00) to (0m1);
\draw[dashed] (00) to (m1m1);
\draw[dashed] (00) to (m2m1);
\draw[dashed] (00) to (m3m1);
\draw[dashed] (00) to (3m1);
  \end{tikzpicture}
\]
Observe that the numbers of $T$- and $R$-cones are unchanged, and the type of the $R$-cone is preserved.
\end{Ex}

\begin{Defn}[\cite{Al-Mo-singularity-content}]\label{Defn:singularity-content}
  Let $P$ be a Fano polygon, let $k$ be the number of $T$-cones in $P$ and $\mathcal{B}$ the list of types of $R$-cones in $P$, ordered cyclically. The set $\mathcal{B}$ is called the \emph{singularity basket} of $P$. The \emph{singularity content} of $P$ is the pair $(k,\mathcal{B})$.
\end{Defn}

Any mutation removes one $T$-cone and adds another, so the total number of $T$-cones is unchanged. The $R$-cones and their relative order is unchanged by mutation, so the singularity content is an invariant under mutation (see \cite{Al-Mo-singularity-content}). If the cyclical order isn't important, it may be useful to think of the singularity basket as a mere multiset; we will do this in Theorem \ref{Thm:monodromy-sing-content}.

\begin{Ex}
  The polygons in example Example \ref{Ex:mutation} have singularity content $(9,\{\frac13(1,1)\})$.
\end{Ex}

\begin{Defn}\label{Defn:f-mutation}
  Let $f$ be a Laurent polynomial with Newton polygon $P$, and let $P'$ be the mutation of $P$ with mutation data $u,F$. The map $\mu:x^a\mapsto x^a(\gamma+\eta x^F)^{\langle u|a\rangle}$ (where $a\in N, \gamma,\eta\in \CC$) defines an automorphism of $\CC(N)$, the rational functions in two variables, and is called a \emph{cluster transformation}. We say that $f$ is \emph{mutable} with respect to $u,(\gamma+\eta x^F)$ if $f'=f\circ\mu$ is in $\CC[N]$ (notice $Newt(\gamma+\eta x^F)=[0,F]$), i.e. is a Laurent polynomial, and in this case that $f'$ is a \emph{mutation of $f$}; the Newton polygon of $f'$ is $P'$. Any two Laurent polynomials related by a chain of mutations are said to be \emph{mutation equivalent}. 

We also say that $f$ is \emph{mutable over} the $T$-cone contracted by the mutation $P\mapsto P'$; we say that a cone over which $f$ is mutable is an \emph{$f$-mutable} cone, and a cone over which $f$ is not mutable is an \emph{$f$-rigid} cone, or if $f$ is understood simply call these \emph{mutable} and \emph{rigid} cones respectively.
\end{Defn}

Let us make explicit what's going on. For $f'$ to be in $\CC[N]$, we require the following: if $P_r$ are the points of $P$ at height $r$ as in Definition \ref{Defn:P-mutation}, let $f_r$ be the terms of $f$ corresponding to the points of $P_r$. To perform a mutation with factor $(\gamma+\eta x^F)$, it is neccessary that $(\gamma+\eta x^F)^r$ is a factor of $f_r$ for $0<r\le h$. The mutated polynomial $f'=mut(f)$ can be described by
\[
f'_r = f_r(\gamma+\eta x^F)^{-r}.
\]
It is clear that $Newt(mut(f))=mut(Newt(f))$. In particular, any mutation of $f$ gives an underlying mutation of $Newt(f)$.

The relevant fact for us is that mutation of $f$ preserves the classical period integral $\pi_f(t)$, and thus the Picard-Fuchs operator $L_f$. A proof of this fact can be found in \cite{ACGK}. The analysis of $\pi_f(t)$ and $L_f$ is then independent of which $f$ in the mutation class we use, which allows for a great deal of flexibility. In particular, we consider the class of \emph{maximally} mutable polynomials. We note that the following definition is a special case only valid for the two-dimensional case; for general dimension a somewhat more involved formulation must be used (see \cite{MaxMutMedAl}). The problems that occur in higher dimensions are not relevant to us, so we keep it simple here.

\begin{Defn}[\cite{MaxMutMedAl}]\label{Defn:max-mut}
  A Laurent polynomial $f$ is called \emph{maximally mutable} if whenever there is a sequence of mutations $Newt(f)=P_0\to P_1\to\cdots\to P_n$, there are Laurent polynomials $f_i$ with $f_0=f$ and $Newt(f_i)=P_i$, such that $f_i$ is mutable over the mutation $P_i\to P_{i+1}$, and $f_{i+1}$ is the resulting mutation of $f_i$.

If in addition $f$ has zero constant term, and for every edge $E$ of $Newt(f)$ of lattice height $h_E$ and lattice width $w_E$, the polynomial $f_E$ is (up to $GL(N)$) equal to $x_1^{h_E}(1+x_2)^{w_E}$ (i.e. $f$ has ``binomial edge coefficients''), $f$ is called \emph{standard maximally mutable}. 

We may for simplicity refer to maximally mutable Laurent polynomials as simply \emph{MMLP's}.
\end{Defn}

The standard MMLP's are of particular importance for the mirror symmetry classification of Fano manifolds. In that literature one usually considers what we call ``standard maximally mutable'' polynomials, with the additional condition that for lattice points on an edge internal to an $R$-cone, the corresponding coefficient is zero---this is called having \emph{$T$-binomial edge coefficients})---for these, the mutations will always be with factor $(1+x^F)$ (e.g. in \cite{overarching,CCGGK,Ale-Andrea}). In the remainder, we will work as generally as possible, but we will return to the specializations of coefficients in the final section.

Requiring that we can mutate $f$ across the whole graph of mutations of $Newt(f)$ means that we must ensure that for every $T$-cone of height $h$, the slices $f_r$ at height $0<r\le h$ must be divisible by some factor $(\gamma+\eta x^F)^r$, or in other words, the maximally mutable Laurent polynomials are those for whom the $f$-mutable cones are the $T$-cones, and the $f$-rigid cones are the $R$-cones. The process of finding the MMLP's for a given polygon $P$ is best illustrated with an example.

\begin{Ex}\label{Ex:max-mutable}
  Let $P$ be the polygon with vertices $(-1,2)$, $(1,2)$, $(2,1)$, $(2,-1)$, $(-2,-1)$ and $(-2,1)$; this has two $R$-cones of type $\frac13(1,1)$ and seven $T$-cones. It is easiest to show the process of finding the maximally mutable Laurent polynomials by labelling the vertices of $P$ by the associated coefficients. 
The cones are indicated; the $R$-cones are shaded grey, the $T$-cones are white. We begin with generic coefficients:
\[
  \begin{tikzpicture}[scale=1.4]
\fill[black!10] (0,0) -- (-2,1) -- (-1,2) -- (0,0);
\fill[black!10] (0,0) -- (2,1) -- (1,2) -- (0,0);

    \node (m12) at (-1,2) {$a_{-1,2}$};
  \node (02) at (0,2) {$a_{0,2}$};
  \node (12) at (1,2) {$a_{1,2}$};
  \node (m21) at (-2,1) {$a_{-2,1}$};
  \node (m11) at (-1,1) {$a_{-1,1}$};
  \node (01) at (0,1) {$a_{0,1}$};
  \node (11) at (1,1) {$a_{1,1}$};
  \node (m20) at (-2,0) {$a_{-2,0}$};
  \node (m10) at (-1,0) {$a_{-1,0}$};
  \node (00) at (0,0) {$a_{0,0}$};
  \node (10) at (1,0) {$a_{1,0}$};
  \node (m2m1) at (-2,-1) {$a_{-2,-1}$};
  \node (m1m1) at (-1,-1) {$a_{-1,-1}$};
\node (21) at (2,1) {$a_{2,1}$};
\node (20) at (2,0) {$a_{2,0}$};
\node (2m1) at (2,-1) {$a_{2,-1}$};
\node (0m1) at (0,-1) {$a_{0,-1}$};
\node (1m1) at (1,-1) {$a_{1,-1}$};

\draw[-] (m12) to (02) to (12) to (21) to (20) to (2m1) to (1m1) to (1m1) to (0m1) to (0m1) to (m1m1) to (m2m1) to (m20) to (m21) to (m12);

\draw[dashed] (00) to (m12);
\draw[dashed] (00) to (12);
\draw[dashed] (00) to (1m1);
\draw[dashed] (00) to (m2m1);
\draw[dashed] (00) to (21);
\draw[dashed] (00) to (0m1);
\draw[dashed] (00) to (2m1);
\draw[dashed] (00) to (m1m1);
\draw[dashed] (00) to (0m1);
\draw[dashed] (00) to (m21);

  \end{tikzpicture}
\]

First impose the factorization conditions along the edges, with a linear factor $(\gamma+\eta x)$ for each $T$-cone. This will determine the ``internal'' coefficients on the edges with $T$-cones of height 2, e.g. $a_{-1,2}+a_{0,2}x+a_{1,2}x^2=(\gamma+\eta x)^2$ (for some $\gamma,\eta$) implies that $a_{0,2}=2(a_{-1,2}a_{1,2})^\frac12$. In the same way, $a_{2,0}=2(a_{2,1}a_{2,-1})^\frac12$ and $a_{-2,0}=2(a_{-2,1}a_{-2,-1})^\frac12$. To reduce visual clutter, we rename the free parameters on the edges by $a,b,c,\ldots$.
\[
  \begin{tikzpicture}[scale=1.4]
\fill[black!10] (0,0) -- (-2,1) -- (-1,2) -- (0,0);
\fill[black!10] (0,0) -- (2,1) -- (1,2) -- (0,0);

    \node (m12) at (-1,2) {$a$};
  \node (02) at (0,2) {$2(ab)^\frac12$};
  \node (12) at (1,2) {$b$};
  \node (m21) at (-2,1) {$i$};
  \node (m11) at (-1,1) {$a_{-1,1}$};
  \node (01) at (0,1) {$a_{0,1}$};
  \node (11) at (1,1) {$a_{1,1}$};
  \node (m20) at (-2,0) {$2(hi)^\frac12$};
  \node (m10) at (-1,0) {$a_{-1,0}$};
  \node (00) at (0,0) {$a_{0,0}$};
  \node (10) at (1,0) {$a_{1,0}$};
  \node (m2m1) at (-2,-1) {$h$};
  \node (m1m1) at (-1,-1) {$g$};
\node (21) at (2,1) {$c$};
\node (20) at (2,0) {$2(cd)^\frac12$};
\node (2m1) at (2,-1) {$d$};
\node (0m1) at (0,-1) {$f$};
\node (1m1) at (1,-1) {$e$};

\draw[-] (m12) to (02) to (12) to (21) to (20) to (2m1) to (1m1) to (1m1) to (0m1) to (0m1) to (m1m1) to (m2m1) to (m20) to (m21) to (m12);

\draw[dashed] (00) to (m12);
\draw[dashed] (00) to (12);
\draw[dashed] (00) to (1m1);
\draw[dashed] (00) to (m2m1);
\draw[dashed] (00) to (21);
\draw[dashed] (00) to (0m1);
\draw[dashed] (00) to (2m1);
\draw[dashed] (00) to (m1m1);
\draw[dashed] (00) to (0m1);
\draw[dashed] (00) to (m21);

  \end{tikzpicture}
\]

We now require the polynomial $i\frac{y}{x^2}+a_{-1,1}\frac{y}{x}+a_{0,1}y+a_{1,1}xy+cx^2y$ along the $y=1$ row to be divisible by $a^\frac12+b^\frac12x$, the polynomial $a\frac{y^2}{x}+a_{-1,1}\frac{y}{x}+a_{-1,0}\frac1x+g\frac1{xy}$ along the $x=-1$ line to be divisible by $h^\frac12+i^\frac12x$, and the polynomial $bxy^2+a_{1,1}xy+a_{1,0}x+e\frac{x}{y}$ along the $x=1$ line to be divisible by $c^\frac12+d^\frac12x$. Solving the equations this imposes, we get 
\begin{itemize}
\item $a_{-1,0} = \frac{h^\frac12}{i^\frac12}a_{-1,1}-\frac{ah}{i}+\frac{g i^\frac12}{h^\frac12}$,
\item $a_{0,1} = \frac{b^\frac12}{a^\frac12} a_{-1,1}+\frac{a^\frac12}{b^\frac12} a_{1,1}-\frac{a c}{b}-\frac{b i}{a}$,
\item $a_{1,0}= \frac{d^\frac12}{c^\frac12} a_{1,1}-\frac{b d}{c}+\frac{c^\frac12 e}{d^\frac12}$.
\end{itemize}

If for simplicity we let $f$ be standard maximally mutable, setting the constant term to zero and imposing binomial edge coefficients, we get the following picture (where we set $a_{-1,1}=p,a_{1,1}=q$ to reduce visual clutter):

\[
  \begin{tikzpicture}[scale=1.4]
\fill[black!10] (0,0) -- (-2,1) -- (-1,2) -- (0,0);
\fill[black!10] (0,0) -- (2,1) -- (1,2) -- (0,0);

    \node (m12) at (-1,2) {1};
  \node (02) at (0,2) {2};
  \node (12) at (1,2) {1};
  \node (m21) at (-2,1) {1};
  \node (m11) at (-1,1) {$p$};
  \node (01) at (0,1) {$\begin{matrix}p+q\\-2\end{matrix}$};
  \node (11) at (1,1) {$q$};
  \node (m20) at (-2,0) {2};
  \node (m10) at (-1,0) {$p+3$};
  \node (00) at (0,0) {0};
  \node (10) at (1,0) {$q+3$};
  \node (m2m1) at (-2,-1) {1};
  \node (m1m1) at (-1,-1) {4};
\node (21) at (2,1) {1};
\node (20) at (2,0) {2};
\node (2m1) at (2,-1) {1};
\node (0m1) at (0,-1) {6};
\node (1m1) at (1,-1) {4};

\draw[-] (m12) to (02) to (12) to (21) to (20) to (2m1) to (1m1) to (1m1) to (0m1) to (0m1) to (m1m1) to (m2m1) to (m20) to (m21) to (m12);

\draw[dashed] (00) to (m12);
\draw[dashed] (00) to (12);
\draw[dashed] (00) to (1m1);
\draw[dashed] (00) to (m2m1);
\draw[dashed] (00) to (21);
\draw[dashed] (00) to (0m1);
\draw[dashed] (00) to (2m1);
\draw[dashed] (00) to (m1m1);
\draw[dashed] (00) to (0m1);
\draw[dashed] (00) to (m21);
\end{tikzpicture}
\]

Observe that there are 12 free parameters $a,b,c,d,e,f,g,h,i,a_{-1,1},a_{1,1}a_{0,0}$ in the coefficients; there is one for each vertex, one for the origin, and one for each point of $P$ not internal to a $T$-cone; in the standard case there are only free parameters corresponding to the internal points in the $R$-cones.
\end{Ex}

This last observation is true in general, and a proof is given in \cite{MaxMutMedAl}.

\begin{Prop}\label{Prop:free-params}
The number of free parameters in a Laurent polynomial with $Newt(f)=P$ is equal to the number of lattice points in $P$ not internal to an $f$-mutable cone. In particular, the number of free parameters in a maximally mutable Laurent polynomial with $Newt(f)=P$ is equal to the number of lattice points in $P$ not internal to a $T$-cone, and the number of free parameters of a standard MMLP is equal to the number of lattice points in $P$ internal to an $R$-cone.
\end{Prop}

\begin{Prop}\label{Prop:mutation-multiple-points}
  Let $P$ be a Fano polygon, and let $f$ be a generic Laurent polynomial with $Newt(f)=P$. Let $P'$ be a mutation of $P$ that contracts a $T$-cone of height $h$ on an edge $E$, and suppose $f$ is mutable over this cone
. Then $f$ has an ordinary multiple point of multiplicity $h$ on $supp(E)$. In particular, a generic maximally mutable Laurent polynomial has a multiple point of multiplicity $h_i$ for each $T$-cone of $P$ of height $h_i$.
\end{Prop}

\begin{proof}
 Recall from Definition \ref{Defn:f-mutation} the conditions for mutability: choose local coordinates $x,y$ so the edge $E$ is contained in the hyperplane $y=h$, and let $f_r$ be the polynomial made up of terms of $f$ corresponding to points at height $r$ (using the same height function). Then in these coordinates, we can write
\[
f_r=(\gamma+ \eta x)^ry^rh_r,
\]
where $h_r=h_r(x)$ is some Laurent polynomial in $x$. Examining in local coordinates, e.g. in the toric chart corresponding to one of the vertices of $E$, where $f$ becomes an honest polynomial, we can easily see that $f$ has an ordinary $h$-uple point here (at the point corresponding to $x=-\gamma/\eta$).
\end{proof}

\begin{Thm}\label{Thm:mutable-genus}
Let $f$ be a generic Laurent polynomial with $Newt(f)=P$. The general fiber $X_t\sse \widetilde{Y_P}$, which is the desingularization of the curve $f=0$ in $Y_P$, has genus equal to the number of internal lattice points of the $f$-rigid cones of $P$, counting the origin. In particular, if $f$ is a generic maximally mutable Laurent polynomial, the genus is the number of internal lattice points of the $R$-cones of $P$, counting the origin, and we call this number the \emph{mutable genus} of $Y_P$ and denote it by $g_{mut}(Y_P)$; it is mutation-invariant.
\end{Thm}

\begin{proof}
Recall that the genus $g(D_P)$ of the desingularization of $D_P$, called the \emph{sectional genus} of $Y_P$, is equal to the number of internal lattice points of $P$ \cite[10.5.8]{Cox-Little-Schenk}. This is the genus of a generic curve in the complete linear system of curves linearly equivalent to $D_P$. The curves defined by Laurent polynomials mutable over a given collection of $T$-cones form a base point-free linear subspace of this linear system in the obvious way, and to find the genus of such a curve, we need to examine how a general Laurent polynomial $f$ with the appropriate mutability differs from a generic section of $D_P$.

It follows from Proposition \ref{Prop:mutation-multiple-points} that a Laurent polynomial has an $h$-uple point for every $T$-cone of height $h$ over which it is mutable;
we have imposed no other conditions, so there are no other special points that affect the genus.

The effect of an ordinary $h$-uple point on the genus of a curve is well known (see e.g. \cite[pp.500-508]{Griffiths-Harris}): the genus drops by $\frac12 h(h-1)$ for every such point. Thus, the genus of the curve defined by $f$ is
$g(D_P)-\sum \frac12 h_i(h_i-1)$, where the sum runs over the $T$-cones of $Newt(f)$ over which $f$ is mutable and the $i$'th cone has lattice height $h_i$.

Now observe that $\frac12 h(h-1)$ is exactly the number of internal lattice points in a $T$-cone of height $h$ (this follows directly from Pick's formula \cite[Ex. 9.4.4]{Cox-Little-Schenk}), so the genus of $f$
is equal to $g(D_P)-\sum \frac12 h_i(h_i-1) = |int(P)\cap N| - |int(P)\cap N\cap \text{$f$-mutable cones}|$, that is, the number of internal lattice points in $P$ that are in $f$-rigid cones, counting the origin. In particular if $f$ is a generic MMLP, it is mutable over all the $T$-cones, so the genus is the number of lattice points in $P$ that are in $R$-cones, counting the origin.

To see that this genus $g$ is mutation-invariant, it is enough to recall that the singularity content of $P$, in particular the set of $R$-cones, is invariant under mutation (\ref{Defn:singularity-content}, also see \cite{Al-Mo-singularity-content}), which of course implies that the number of internal lattice points in the $R$-cones is invariant; in particular it is preserved by those mutations of $P$ over which $f$ is mutable.
\end{proof}

\begin{Rem}
  We remark that the genus is unchanged even if some the multiple points on the $T$-cones on the same edge over which the polynomial is mutable are allowed to coincide, i.e. if $f$ is mutable with the same factor on all the cones. This is because when we deform $k$ ordinary $h$-uple points to coincide, the result is not an ordinary $kh$-uple point, but an $h$-uple point where the branches meet with an order $k$ tangency (the case $k=2,h=2$ is the familiar tacnode); an order $k$ tangency will drop the genus by $k$ for every branch, so the total defect is still $k\cdot \frac12 h(h-1)$ \cite[pp.500-508]{Griffiths-Harris}. Note also that the standard MMLP's may have genus lower than $g_{mut}(Y_P)$, as fixing a coefficient 0 at the origin may cause problems. An example is the pictured polygon (the internal point is the origin):
\[  
\begin{tikzpicture}[scale=1]
\draw[step=1,gray,dashed,very thin] (-1.4,-1.4) grid (1.4,1.4);
\node (m11) at (-1,1) {$\bullet$};
\node (01) at (0,1) {$\bullet$};
\node (11) at (1,1) {$\bullet$};
\node (00) at (0,0) {$\bullet$};
\node (0m1) at (0,-1) {$\bullet$};
\draw[-] (m11) -- (11) -- (0m1) -- (m11);
\end{tikzpicture}
\]
Fixing binomial edge coefficients we get a polynomial $f=\frac{y^2}{x}+2y^2+xy^2+\frac1y + a$; the curve $f=0$ has genus 1 unless $a=\pm 4$ or $a=0$, in which case it has genus 0. 
\end{Rem}

Recall that we can write $L_f=\sum_rp_r(t)\nabla^r$. The \emph{order} of $L_f$ is the maximal $r$ occurring in the sum, and the \emph{degree} of $L_f$ is the maximal degree (in $t$) of the $p_r$'s. The degree is hard to say anything about (but see Section \ref{sec:ramification}), however the order is now available to us:

\begin{Cor}\label{Cor:order-of-L_f}
Let $f$ be a Laurent polynomial with $Newt(f)=P$. Then the order of the Picard-Fuchs operator $L_f$ is 2 times the genus of $X_t$; in particular if $f$ is a generic maximally mutable Laurent polynomial, the order of $L_f$ is $2g_{mut}(Y_P)$.
\end{Cor}

\begin{proof}
It follows from the Cauchy-Kovalevski theorem (\cite[9.4.5]{Horm}) that the order of $L_f$ is equal to the rank of its solution space, and it is a well-known fact that $H_1(X,\CC)\simeq \CC^{2g}$ if $X$ is a compact Riemann surface of genus $g$.
\end{proof}

%
%

\section{Monodromy at $t=0$}\label{sec:monodromy}

To compute the monodromy of $H_1(X_t,\CC)$, we need to find a suitable basis of cycles, and a description of the monodromy automorphism. We will do this by explicitly constructing a model for $X_t$ by means of local calculations, explicitly carrying out the resolution $\widetilde{Y_P}\to Y_P$.

Let us recap what we know so far: The general fiber $X_t\sse \widetilde{Y_P}$ is a genus $g_{mut}$ curve, which degenerates as $t\to 0$ to the support of the divisor $D$, the pullback of $D_P$ to $\widetilde{Y_P}$. This divisor is in any case a collection of $\PP^1$'s, topologically a necklace of spheres, with some chains of spheres attached (each sphere corresponds to an edge of $P$ or an exceptional curve of the resolution $\widetilde{Y_P}\to Y_P$). Recall from Theorem \ref{Thm:mutable-genus} that $g_{mut}$ is equal to the number of internal lattice points of $P$ that are not internal to an $f$-mutable cone, which always includes the origin as $P$ by assumption is Fano. A necklace of spheres is a degeneration of a topological surface of genus at least one, which would account for the contribution to the genus from the lattice point at the origin. By \ref{Thm:mutable-genus} the rest of the genus comes from the internal points of the $R$-cones of $P$, so there must be some singularities on the $\PP^1$'s corresponding to the edges that resolve to give a higher-genus surface.

We may thus reduce to a series of local considerations, which we will refer to as the \emph{contributions} from the vertices and edges, and $f$-rigid cones respectively. The contribution from the vertices is this: intersection points between the components are degenerations of the form $\{x^my^n=t\}\to\{x^my^n=0\}$, and we must describe which of these occur and what the monodromy does to them (see Figure \ref{fig:vanishing-cycle}). The contributions from the $R$-cones is this: on the components of $D_P$ corresponding to edges with $f$-rigid cones, we must identify what singular points occur and resolve them to get a positive-genus curve $\widetilde{C}\to \PP^1$; then find an appropriate automorphism of $\widetilde{C}$ that fixes the inverse images of all the singular points and intersection points with the adjoining components of $D_P$ (see Figure \ref{fig:genus}).

\begin{figure}
\centering
  \includegraphics[width=0.5\textwidth]{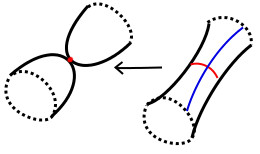}
  \caption{Local picture of the degeneration over an intersection between components of $D_P$; the vanishing cycle is indicated in red, and the relative cycle in blue.}\label{fig:vanishing-cycle}
  
\end{figure}

\begin{figure}
\centering
  \includegraphics[width=0.7\textwidth]{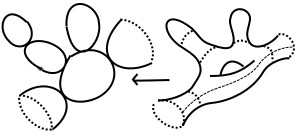}
  \caption{The component of $D$ corresponding to an $f$-rigid cone, showing the exceptional curves of some resolved singular points; this is a degeneration of a higher-genus surface, vanishing and relative cycles indicated. Notice how the monodromy automorphism of $X_t$ must fix these vanishing cycles to degenerate correctly to the special fiber.}\label{fig:genus}
  
\end{figure}

To compute the whole monodromy action on $X_t$, we will then cut the curve into pieces and consider each piece by itself, and then assemble the results afterwards. Some of the basis cycles of $H_1(X_t,\CC)$ will exist entirely within these pieces (that is, they are homologous to cycles contained in the local piece), these will be cycles that degenerate to a point in the special fiber, and are as such called \emph{vanishing cycles}. The remaining cycles will in the local pictures enter and exit the local piece through the cuts, these will be called \emph{relative cycles} in the local pictures.

We fix some notation: There is in fact only a single such global cycle that locally becomes a relative cycle, from here on we will call this cycle (and its local images) $\alpha$. The vanishing cycles over the intersections between components of $D_P$ are all homologous, and we will call this cycle $\beta$.

\begin{figure}
\centering
  \includegraphics[width=0.7\textwidth]{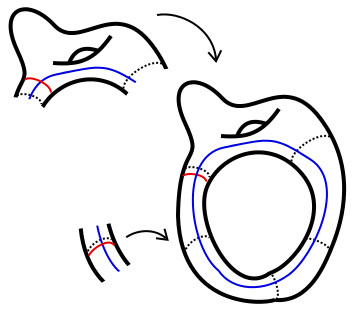}
  \caption{The cycle $\alpha$ marked in blue and the vanishing cycle $\beta$ marked in red, and their images in the local pieces.}\label{fig:global-cycles}
\end{figure}

Observe that the local monodromy action on the relative cycles need not be integral, as long as these globally add up to something integral (indeed, exploiting this fact will be crucial in some of the local calculations). 

\subsection{The singularities of $Y_P$, and intersections between the components}\label{subsec:iscts}

After we have resolved the singularities of $Y_P$, we may look at the monodromy action over the intersections between the components of $D$. Locally at the intersection between two components of $D$, of multiplicities $m$ and $n$ respectively, in the local coordinates given by the toric chart corresponding to the vertex of intersection, we can write $supp(D)$ as $\{x^my^n=0\}$. In these local coordinates, the global sections $1$ and $f$ become $x^my^n$ and $1+(\text{higher-order terms})$ respectively, and we can write $1-tf$ as $x^my^n-t(1+(\text{higher-order terms}))$, locally analytically equivalent to $x^my^n-t$. The degeneration when $t\to 0$ is now equivalent to $\{x^my^n=t\}\to\{x^my^n=0\}$. The monodromy action is then locally the monodromy of the curve $x^my^n=t$ as $t$ goes around zero. 

\begin{Lemma}\label{Lemma:basic-monodromy}
 Let $\beta$ be the vanishing cycle of $x^my^n=t$ when $t\to 0$ (with positive orientation), and let $\alpha$ be the relative cycle. The monodromy action on $\alpha,\beta$ in $x^my^n=t$ as $t$ goes around $t=0$ in the positive direction is given by $\beta\mapsto \beta,\alpha\mapsto \alpha-\frac{1}{mn}\beta$.
\end{Lemma}

\begin{proof}
  Consider the Riemann surface of $y=\sqrt[n]{\frac{t}{x^m}}$, for fixed $t$. This is an $n$-sheeted covering of the punctured complex plane with a singularity at $x=0$, where as you trace along the surface around the singularity, $y$ will alternate between approaching $+\infty$ and $-\infty$ as $x$ approaches zero, alternating a total of $m$ times (see Figure \ref{fig:32-monodromy} for a picture of what this looks like). Notice the $m$-fold rotational symmetry of the surface.

Write $t=e^{i\theta}$ and $x=e^{i\tau}$ (we may ignore the magnitude as only the argument is relevant to the monodromy action); we may now express the surface as
\[
y=(tx^{-m})^{\frac1n} = (e^{i(\theta-m\tau)})^{\frac1n}.
\]
In other words, when $t$ moves around the origin, the resulting surface satisfies an equation $y=(x_\theta^{-m})^{\frac1n}$, where $x_\theta=e^{i\tau(\theta)}$, and the argument satisfies $-m\tau(\theta)=\theta-m\tau$. From this, $\tau(\theta)=-\theta/m+\tau$, we see that the surface will rotate in the same direction as $t$, with $\frac1m$'th the speed. Thus, when $t$ has completed a full revolution, the surface will have rotated by an angle of $\frac{2\pi}m$, or one step along the $m$-fold rotational symmetry. 

To find the effect of this on the cycles $\alpha$ and $\beta$, we give an explicit model for each. The vanishing cycle $\beta$ is homologous to the curve $\{(e^{i\theta},e^{-\frac{i\theta}n})|0\le\theta\le 2n\pi\}$ that winds around the singularity $n$ times, following the sheets until it meets itself. This curve is preserved under the rotational symmetry of the surface, so the monodromy action on $\beta$ is the identity. The relative cycle $\alpha$ can be modelled by a curve going along the topmost sheet of the surface from $(\eps,\eps^{-\frac{m}{n}})$ to $(K,K^{-\frac{m}{n}})$, where $\eps\ll 1$ and $K\gg 1$ are real numbers (note the orientation). The monodromy action can be modelled by pinning the initial point $(\eps,\eps^{-\frac{m}{n}})$ in place (i.e. letting it rotate along with the surface) while holding the other fixed over $x=K$. After the monodromy action, the inital point has been moved to $(\eps\cdot e^{2\pi i/m},\eps^{-\frac{m}{n}}e^{2\pi i/m})$, while the final point, fixed to lie over $x=K$, will be on the sheet immediately below the topmost one. The resulting curve is homologous to $\alpha-\frac{1}{mn}\beta$, as the $m$-fold rotational symmetry moves a point $\frac1{mn}$'th of the length of $\beta$.
\end{proof}

\begin{figure}
\centering  
\includegraphics[width=0.5\textwidth]{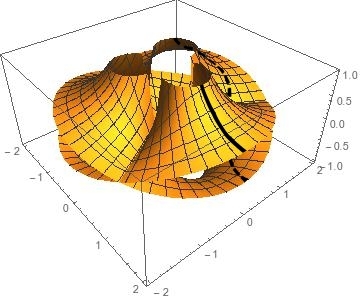}
  \caption{The Riemann surface of $y=Re(\sqrt{\frac{1}{x^3}})$; the solid curve is the relative cycle $\alpha$, and the vanishing cycle $\beta$ can be identified with the outer boundary of the displayed surface. The dashed curve is the cycle $\alpha-\frac16\beta$.}\label{fig:32-monodromy}
\end{figure}

We now find the pullback of $D_P$ when resolving the singularities of $Y_P$, this will give us all the points on $X_0$ that locally are of the form $x^my^n=0$, and now \ref{Lemma:basic-monodromy} tells us what the local monodromy action is. Notice that we can combine the local actions without a problem, as the vanishing cycles $\beta$ appearing in all of them are homologous, so the local actions commute.

Recall that $D_P=\sum h_iD_i$, where $h_i$ are the lattice heights of the edges $E_i$ of $P$ corresponding to the divisors $D_i$. The resolved divisor $D$ can be written $D=D_P+\sum m_jF_j$, where $F_j$ are some exceptional curves and $m_j$ are their multiplicities. An interesting fact is that the numbers $m_j$ are such that it makes sense to think of the $F_j$ as corresponding to ``edges'' of $P$ of width zero and height $m_j$; we will however not need this. 

Suppose now $v$ is a vertex of $P$, corresponding to a cone in the normal fan where $Y_P$ has a singularity of type $\frac1r(1,a)$, and that $v$ is joining edges $E$ and $E'$, of heights $h$ and $h'$. The singularity is resolved according to \cite[Chapter 10]{Cox-Little-Schenk}; recall in particular the notion of Hirzebruch-Jung continued fractions, denoted as follows:
\[
[b_1,b_2,\ldots,b_k]=b_1-\frac1{b_2-\frac{1}{\ldots-\frac1{b_k}}}.
\]
We introduce some notation: suppose the Hirzebruch-Jung continued fraction expansion of $r/a$ is $[b_1,\ldots,b_k]$; let $s_1=t_k=1$, and define positive integers $s_i,t_i$ by
\begin{align*}
  s_i/s_{i-1}:=[b_{i-1},\ldots,b_1],&\quad  2\le i\le k\\
  t_i/t_{i+1}:=[b_{i+1},\ldots,b_k],&\quad  1\le i\le k-1.
\end{align*}
Note that we may extend this to letting $s_0=t_{k+1}=0$ and $s_{k+1}=t_0=r$.

When resolving the singularity at $v$, we get $k$ exceptional curves $F_1,\ldots F_k$, with self-intersections $F_i^2=-b_i$. Let $m_i$ denote the multiplicity of $E_i$ in $D$; these multiplicities are determined by the criterion that $E_i.D=0$.

\begin{Lemma}\label{Lemma:blowup-numbers}
  \begin{enumerate}
  \item $s_{i+1}+s_{i-1}=b_is_i$ and $t_{i+1}+t_{i-1}=b_it_i$.
\item $m_i = \frac1r (t_im_0+s_i m_{k+1})$.
\item $m_0 = s_{i+1}m_i-s_im_{i+1}$.
  \end{enumerate}
\end{Lemma}

\begin{proof}
  \emph{1.}  By definition, $s_{i+1}/s_i = [b_i,\ldots,b_1] = b_i-1/[b_{i-1},\ldots,b_1] = b_i-s_{i-1}/s_i$ and it follows that $s_{i+1}+s_{i-1}=b_is_i$, a similar rearrangement shows the other identity.

\emph{2.} Recall that the $m_i$ are defined by the system of equations $E_i.D=0$. As the only components of $D$ that are involved are $D_F$, $D_{F'}$ and the $E_i$'s, and the intersection numbers are 1 for adjacent components and 0 for non-adjacent components, we get equations
 $m_{i-1}-b_im_i+m_{i+1}=0$ (for $1\le i\le k$). Successive elimination, applying item 1 at each step, now yields the desired conclusion.

\emph{3.} We show this by induction. The base case is the equation $m_{i-1}-b_im_i+m_{i+1}=0$ for $i=1$, using that $s_1=1$ and $s_2=b_1$. The induction step is to show that $s_{i+1}m_i-s_im_{i+1}=s_{i+2}m_{i+1}-s_{i+1}m_{i+2}$; rearranging we have $s_{i+1}m_i+s_{i+1}m_{i+2}=s_{i+2}m_{i+1}+s_im_{i+1}$, and applying the identity $s_{i+2}+s_{i}=b_{i+1}s_{i+1}$ on the right-hand side and the equation $m_{i}+m_{i+2}=b_{i+1}m_{i+1}$ on the left-hand side we see that both sides equal $b_{i+1}s_{i+1}m_{i+1}$.
\end{proof}

\begin{Lemma}\label{Lemma:r/mn}
$\sum_{i=0}^k \frac{1}{m_im_{i+1}} = \frac{r}{m_0m_{k+1}}$.
\end{Lemma}

\begin{proof}
  Observe first that $\frac1{m_0m_1}+\frac1{m_1m_2} = \frac1{m_1} \frac{m_2+m_0}{m_0m_2}$, and by \ref{Lemma:blowup-numbers}(3) $m_2+m_0 = s_1m_2+m_0 = s_2m_1$, so we get $\frac1{m_0m_1}+\frac1{m_1m_2} = \frac{s_2}{m_0m_2}$. In similar fashion we see that  $\frac{s_i}{m_0m_i}+\frac1{m_im_{i+1}} = \frac1{m_i}\frac{s_im_{i+1}+m_0}{m_0m_{i+1}} =\frac{s_{i+1}}{m_0m_{i+1}}$, so by induction we have $\sum_{i=0}^k \frac{1}{m_im_{i+1}} = \frac{s_{k+1}}{m_0m_{k+1}} = \frac{r}{m_0m_{k+1}}$.
\end{proof}

\begin{Prop}\label{Prop:degree-contr}
  The contribution to the global monodromy of the cycles $\alpha,\beta\in H_1(X_t)$ from the vertices of $P$ is $\alpha\mapsto \alpha - (K_\Pi^2)\beta$ and $\beta\mapsto\beta$, where $\Pi$ is the toric variety defined by the spanning fan of $P$ and $K_\Pi$ is its canonical divisor.
\end{Prop}

\begin{proof}
Combining \ref{Lemma:basic-monodromy}, \ref{Lemma:blowup-numbers} and \ref{Lemma:r/mn} tells us that the contribution from a vertex of $P$ is $\alpha\mapsto \alpha - \frac{r}{mn}\beta$, where $m,n$ are the lattice heights of the adjoining edges and the singularity of $Y_P$ in the corresponding chart is of type $\frac1r(1,a)$ (or if $Y_P$ is smooth here, take $r=1$).

  It is well-known that $K_\Pi^2$ is equal to the lattice volume of the dual polytope $P^\circ\sse M_\RR$ of $P$ (see \cite[13.4.1]{Cox-Little-Schenk}). To show the claim, it is enough to show that the volume of the cone $C_v$ in $P^\circ$ corresponding to the vertex $v$ of $P$ is equal to $\frac{r}{mn}$.
Let $u,w$ be primitive lattice generators of $C_v$.
By Definition \ref{Defn:type-of-cones}, $C_v$ is of type $\frac1r(1,a)$ when $\{\frac1r u+\frac{a}{r} w,w\}$ is a lattice basis for $M$. As $\{\frac1r u+\frac{a}{r} w,w\}$ is a lattice basis, we have $\det(\frac1r u+\frac{a}{r} w,w)=1$, and it follows that $\det(u,v)=r$. Observing now that $C_v$ is spanned by $\frac1m u$ and $\frac1n w$, we are done as $\det(\frac1m u,\frac1n w)=\frac{r}{mn}$.
\end{proof}

\subsection{Monodromy over an $f$-rigid cone}\label{subsec:R-cone-monodromy}

We know from Theorem \ref{Thm:mutable-genus} that the $f$-mutable cones do not contribute to the genus of $X_t$, and so we may ignore them for purposes of the monodromy computation. Assume we have an edge $E$ of $P$ supporting a single $f$-rigid cone, of height $h$ and width $w$, and denote by $X_E$ the inverse image under the map $X_t\to X_0$ of the strict transform of $D_E$ under the resolution $\widetilde{Y_P}\to Y_P$. The strict transform of $D_E$ is a $\PP^1$, and it intersects the adjacent components of the pullback of $D_P$, as well as the exceptional curves coming from resolving the singularities of $f$ on $D_E$. The inverse image $X_E$ is then the part of $X_t$ bounded by the vanishing cycles over these points of intersection. From this and \ref{Thm:mutable-genus} we see that topologically, $X_E$ is a surface with a number of punctures, one for each of these vanishing cycles, with genus equal to $\frac12(h-1)w$, the number of internal points in the $f$-rigid cone (this follows directly from Pick's formula). 

The terms of $f$ along the edge $E$ can be written as $y^h\prod_{i=1}^w(x-\eta_i)$ in suitable coordinates. Now locally at each $\eta_i$ it is not hard to see that $1-tf$ is analytically equivalent to $y^h-t(x+y)$, an $A_{h-1}$-singularity. Indeed, the ``suitable coordinates'' just referred to is the chart of $Y_P$ corresponding to one of the vertices of $E$; in these coordinates, 1 becomes $x^ey^h$ (where $e$ is some positive integer, for now it isn't important which) and $f$ is an ordinary polynomial in $x,y$. The variable change $x\mapsto x-\eta_i$ is best described pictorially by looking at the effect of the Newton polytope of $1-tf$; in the picture circles represent points at height 0 (relative to $t$), while everything else is at height 1. Note that $x^ey^h\mapsto (x-\eta_i)^ey^h$ under this variable change, which changes the one point at height zero (the 1 in $1-tf$) to several points. The points now visible from the origin are $(0,h,0)$, $(0,1,1)$ and $(1,0,1)$, so $1-tf$ is analytically equivalent to $y^h-t(x+y)$ at the singular point.
\[
\begin{tikzpicture}[scale=0.7]
\fill[black!05] (-2,1) -- (-2,0) -- (-2,-4) -- (2,-4) -- (3,-3) -- (4,1);
\node (00) at (0,0) [shape=circle,draw] {};
\node (ypivot) at (-2,0) {$\bullet$};
\node (origin) at (-2,-4) {$\bullet$};
\node (highery) at (-2,1) {};
\node (R-end) at (2,-4) {$\bullet$};
\node (higherx) at (3,-3) {};
\draw[-] (highery) -- (ypivot) to node {$h 
  \begin{cases}
    & \\
& \\
& \\
& \\
& \\
  \end{cases}
  \text{ \emph{} }$} (origin) to node {$\begin{matrix} \\ w\end{matrix}$} (R-end) -- (higherx);

\node (arrowstart) at (5,-1) {};
\node (arrowend) at (7,-1) {};
\draw[|->] (arrowstart) to (arrowend) {};

\fill[black!05] (8,1) -- (8,0) -- (8,-3) -- (9,-4) -- (12,-4) -- (13,-3) -- (14,1);

\node (00new) at (10,0) [shape=circle,draw] {};
\node (underypivotnew) at (8,0) [shape=circle,draw] {};
\node (ypivotnew) at (8,0) {$\bullet$};
\node (morey) at (9,0) [shape=circle,draw] {};
\node (originnew) at (8,-4) {$\cdot$};
\node (higherynew) at (8,1) {};
\node (cone-endnew) at (12,-4) {$\bullet$};
\node (R-endnew) at (9,-4) {$\bullet$};
\node (higherxnew) at (13,-3) {};
\node (breakpoint) at (8,-3) {$\bullet$};
\draw[-] (higherynew) -- (ypivotnew) -- (breakpoint) -- (R-endnew) -- (cone-endnew) -- (higherxnew);
\node (heightlabel) at (7.5,-3.5) {$1\Big\{$};

\end{tikzpicture}
\]
Thus, we have on $D_E$ the $w$ singular points of $f$, each of type $A_{h-1}$, and the two points of intersection with the adjacent components of $D_P$. As $D_E$ has multiplicity $h$, we can now model our $X_E$ as a ramified degree $h$ cover of $\PP^1$, with two ramification points of ramification index $h$ (corresponding to the intersection points with the adjacent divisors), and $w$ ramification points corresponding to the singular points of $f$. More precisely, $X_E$ is homotopic to such a surface, punctured at the two ramification points over the intersections with the adjacent components. The ramification index $e_p$ of these points is found by the Riemann-Hurwitz formula: setting $g=\frac12(h-1)w$ in
\[
2g-2 = -2h +2(h-1) + w(e_p-1)
\]
gives $e_p=h$, so we have a degree $h$ map, ramified at $w+2$ points of ramification index $h$. 

The local monodromy action on $H_1(X_E)$ must then be induced by an automorphism of $X_E$ with $w+2$ fixed points, near which the automorphism has order $h$ (to be compatible with the ramification index); this implies the automorphism has order $h$ everywhere (\emph{a priori} it has order a factor of $h$). We have shown:

\begin{Lemma}\label{Lemma:R-cone-monodromy}
  Let $E$ be a edge of $P$ with an $f$-rigid cone of height $h$ and width $w$. Then the local monodromy action on $H^1(X_E)$, where $X_E$ is as above, is given by an order $h$ automorphism of a genus $\frac12(h-1)w$ surface with $w+2$ fixpoints.
\end{Lemma}

There are of course many such automorphisms, but we can limit the possibilities to some extent, by taking advantage of the fact that a Riemann surface is a quotient of a plane (the Euclidean plane if $g=1$, the hyperbolic plane if $g\ge 2$) by a suitable lattice; any order $h$ automorphism of the surface descends from an order $h$ automorphism of the plane. The two cases must be treated separately, but are ultimately quite similar. The fact that the global monodromy of $H_1(X_t)$ is integral (see \cite[5.4.32]{Zoladek}) gives another restriction; we already know (Proposition \ref{Prop:degree-contr}) the total contribution from the vertices of $P$, and whatever numbers we get from the contribution of the $f$-rigid cones must fit with this to produce something integral.

Recall from \ref{Prop:degree-contr} that the total monodromy from the vertices of $P$ is equal to the degree $K_\Pi^2$ of the toric variety defined by the spanning fan of $P$. We may reformulate this to a count of contributions from the edges by using a result of Akhtar and Kasprzyk (\cite[Prop. 3.3]{Al-Mo-singularity-content}). Recall also from section \ref{subsec:iscts}, for a singularity $\sigma$ of type $\frac1r(1,a)$, the numbers $b_1,\ldots,b_{k_\sigma}$ making up the Hirzebruch-Jung continued fraction expansion of $r/a$, and the numbers $s_i,t_i$ ($1\le i\le k_\sigma$) defined in terms of the $b_i$. Let also $d_i=(s_i+t_i)/r -1$, and let $A(\sigma)=k_\sigma+1-\sum_{i=1}^{k_\sigma}d_i^2b_i+2\sum_{i=1}^{k_\sigma-1}d_id_{i+1}$. Note that if $\sigma$ is a primitive $T$-cone, $A(\sigma)=1$.

\begin{Prop}[Akhtar-Kasprzyk,\cite{Al-Mo-singularity-content}]\label{Prop:degree-formula}
  Let $\Pi$ be a complete toric surface with singularity content $(n,\mathcal{B})$. Then 
\[
K_\Pi^2= 12-n-\sum_{\sigma\in \mathcal{B}}A(\sigma).
\]
\end{Prop}

This and \ref{Prop:degree-contr} together give that for each $R$-cone of type $\sigma$, the contribution to the total monodromy of the relative cycle $\alpha$ from the vertices is $\alpha\mapsto \alpha+A(\sigma)\beta$. The number $A(\sigma)$ is in general not an integer, and so the action on $\alpha$ in the local monodromy of the corresponding $H_1(X_E)$ must be of the form $\alpha\mapsto (\text{something})+B(\sigma)\beta$, where $A(\sigma)+B(\sigma)$ is an integer.

We must separate the cases of $g=1$ and $g\ge 2$; in the first case the surface is a quotient of the Euclidean plane by a lattice, and so its automorphism group is isomorphic to $GL_2(\ZZ)$, in the second case the surface is a quotient of the hyperbolic plane, so its automorphism group is isomorphic to a subgroup of $PSL_2(\CC)$. In either case, to have an order $h$ automorphism, the eigenvalues must be $h$'th roots of unity. 

\subsubsection{The case of genus 1}

For genus 1, note that requiring $g=\frac12(h-1)w=1$ implies that either $h=3$ and $w=1$, or $h=w=2$. Thus, the only possible $f$-rigid cones giving a genus one surface are the cones of type $\frac13(1,1)$ and $\frac14(1,1)$; we should have $w+2$ fixpoints, so we want an order 3 automorphism with 3 fixpoints, or an order 2 automorphism with 4 fixpoints.

\begin{Lemma}\label{Lemma:finite-order-GL2}
  Let $A\in GL_2(\ZZ)$. Then
  \begin{enumerate}
  \item if $A^2=I$ and $A$ has exactly 4 fixpoints modulo $\ZZ^2$, then $A=-I$ (here $I$ is the $2\times 2$ identity matrix), and the fixpoints are of the form $(\frac{n}2,\frac{m}2)$ with $n,m=0$ or $1$.
\item if $A^3=I$ and $A$ has exactly 3 fixpoints modulo $\ZZ^2$, then the fixpoints are of the form $(\frac{n}3,\frac{m}3)$ with $n,m=0,1$ or 2, and $A$ is similar to the matrix $\begin{pmatrix}0&-1\\1&-1\end{pmatrix}$.
  \end{enumerate}
\end{Lemma}

\begin{proof}
  For (1), it is easy to compute that a $2\times 2$ integer matrix with order 2 is one of the following:
  \begin{itemize}
  \item $\pm I$, or
\item $\begin{pmatrix}a&b\\\frac{1-a^2}{b}&-a\end{pmatrix}$ for $a,b\in \ZZ$ such that $b|1-a^2$.
  \end{itemize}
The reader can now easily verify that of these, only $-I$ has exactly 4 fixpoints modulo $\ZZ^2$.

For (2), observe that order three implies that the eigenvalues must the two primitive third roots of unity, and such a matrix by necessity has determinant 1 and trace -1. Imposing on a general integer matrix $\begin{pmatrix}a&b\\c&d\end{pmatrix}$ these conditions gives a matrix of the form 
\[
\begin{pmatrix}a&b\\ -\frac{a^2+a+1}{b} & -(a+1)\end{pmatrix}
\]
(we may safely assume $b\neq 0$, as assuming either $b=0$ or $c=0$ from the beginning yields a noninteger matrix, because the only solutions to $a+d=-1, ad=1$ are the two primitive third roots of unity). We now find conditions on possible fixpoints: if $(x,y)$ is a fixpoint modulo $\ZZ^2$, that requires
\begin{align*} 
x &\equiv_\ZZ a x+b y\\ 
y &\equiv_\ZZ -\frac{a^2+a+1}{b} x -(a+1)y
\end{align*}
which after some simplification gives $3x\equiv_\ZZ 0$ and $3y\equiv_\ZZ0$. Any fixpoints are thus of the form $(\frac{n}3,\frac{m}3)$, for $n,m=0,1,2$. Now, if $(\frac{n}3,\frac{m}3)$ is fixed, so too is $(\frac{2n}3,\frac{2m}3)$, as 2 and 3 are coprime, so if we want exactly three fixpoints, they are either $(0,0),(\frac13,\frac13),(\frac23,\frac23)$, or $(0,0),(\frac13,\frac23),(\frac23,\frac13)$. These configurations are mirror images of each other, so they are congruent.

By conjugating with elementary matrices, we see that the matrix of the above form with parameters $(a, b)$ is similar to those with parameters $(a\pm b, b)$, $(a, -b)$, $(-a-1, -(a^2+a+1)/b)$, $(a+(a^2+a+1)/b, b+(a^2+a+1)/b+2a+1)$ or $(a-(a^2+a+1)/b, b+(a^2+a+1)/b-2a-1)$. Iterating application of these similarities we can eventually arrive at $(0,-1)$; this is because $a,b$ and $(a^2+a+1)/b$ are necessarily coprime.
\end{proof}

\begin{Prop}\label{Prop:1411-local-monodromy}
  Let $X$ be a Riemann surface of genus one, possessing an order two automorphism $\omega$ with exactly four fixpoints $p_1,p_2,q_1$ and $q_2$, and let $X'$ be $X$ punctured at $p_1$ and $p_2$. On $X'$, let $\alpha$ be a relative cycle passing from $p_1$ to $p_2$, and let $\beta$ be a cycle going once around $p_1$. Then there is a choice of cycles $a_1,a_2$ such that $\{\alpha,\beta,a_1,a_2\}$ is a basis for $H_1(X')$, and the automorphism $\omega$ of $X'$ has an induced action on homology given (in this basis) by the matrix
\[
\begin{pmatrix}
  1&0&0&0\\
1&1&0&0\\
-1&0&-1&0\\
-1&0&0&-1
\end{pmatrix}.
\]
\end{Prop}

\begin{proof}From \ref{Lemma:finite-order-GL2} we have that the automorphism must be $-I$ (modulo $\ZZ^2$), and the fixpoints must be $(0,0),(\frac12,0),(0,\frac12)$ and $(\frac12,\frac12)$. Changing basis in $\ZZ^2$ if necessary, we may assume $p_1=(0,0), p_2=(\frac12,\frac12), q_1=(\frac12,0)$ and $q_2=(0,\frac12)$. Now take the basis cycles $a_1,a_2$ to be the edges of the fundamental domain. It is now clear that $\omega(a_i)=-a_i$ and $\omega(\beta)=\beta$, and observe (see Figure \ref{fig:1411}) that $\omega(\alpha)-\alpha+\frac22\beta$ is homologous to $-a_1-a_2$.
  \begin{figure}
\centering
    \begin{tikzpicture}[scale=5,>=stealth]
      \node (p1) at (0,0) {$p_1$};
\node (p2) at (1/2,1/2) {$p_2$};
\node (p1-2) at (1,0) {$p_1$};
\node (p1-3) at (0,1) {$p_1$};
\node (p1-4) at (1,1) {$p_1$};
\node (q1) at (.5,-0.05) {$q_1$};
\node (q2) at (-0.05,.5) {$q_2$};
\draw[->] (p1) to node [swap] {} (p1-2);
\node (a1) at (.7,-.1) {$a_1$};
\node (a1-2) at (.7,1.1) {$a_1$};
\node (a2) at (-.1,.7) {$a_2$};
\node (a2-2) at (1.1,.7) {$a_2$};
\draw[->] (p1) to node [swap] {} (p1-3);
\draw[->] (p1-3) to node {} (p1-4);
\draw[->] (p1-2) to node {} (p1-4);
\draw[->] (p1) to node {} (p2);
\node (alphalabel) at (.25,.3) {$\alpha$};
\draw[->] (p1-4) to node {} (p2);
\node (omegalabel) at (.7,.8) {$\omega(\alpha)$};
\draw[->] (.6,.6) arc (45:-135:0.15);
\node (beta) at (.6,.3) {$\frac12 \beta$};
\draw[->] (.1,.1) arc (45:180:.15);
\draw[->] (1,1.15) arc (90:225:.15);
\node (beta1) at (-.15,.15) {$\frac12\beta$};
    \end{tikzpicture}
    \caption{The automorphism from Proposition \ref{Prop:1411-local-monodromy}.}\label{fig:1411}
    
  \end{figure}
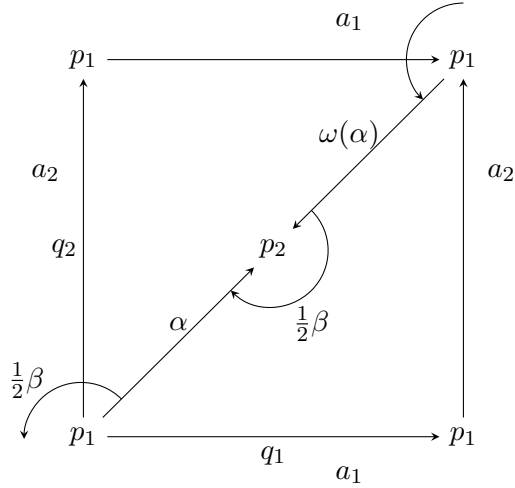
\end{proof}

\begin{Prop}\label{Prop:1311-local-monodromy}
  Let $X$ be a Riemann surface of genus one, possessing an order three automorphism $\omega'$ with three fixpoints $p_1,p_2$ and $q$, and let $X'$ be $X$ with $p_1,p_2$ removed. On $X'$, let $\alpha$ be a relative cycle passing from $p_1$ to $p_2$, and let $\beta$ be a cycle going once around $p_1$. Then there is a choice of cycles $a_1,a_2$ such that $\{\alpha,\beta,a_1,a_2\}$ is a basis for $H_1(X)$, and the automorphism $\omega'$ of $X$ gives an induced action on homology given (in this basis) either by the matrix
\[
\begin{pmatrix}
  1& 0 & 0 & 0\\
\frac13&1&0&-1\\
0&0&0&-1\\
-1&0&1&-1
\end{pmatrix}
\]
or by its inverse.
\end{Prop}

\begin{proof}From \ref{Lemma:finite-order-GL2} we have that the automorphism is similar to the one induced by $\begin{pmatrix}0&-1\\1&-1\end{pmatrix}$; let $a_1,a_2$ be the basis cycles in which $\omega'$ has this representation; we may choose $q=(0,0),p_1=(\frac13,\frac13)$ and $p_2=(\frac23,\frac23)$.

The action on the basis $\{\alpha,\beta,a_1,a_2\}$ can be visualized by constructing a model for $X$ as follows: divide the fundamental domain in $\CC^2$ along one diagonal to form two triangles; label the sides of the fundamental domain by $a_1,a_2$, and label the diagonal by $a_3$, with orientations as indicated in Figure \ref{fig:1311-automorphism}. Let $p_1,p_2$ be the centroids of the triangles, and let $q$ be the origin. We choose two of the cycles $a_i$ as basis cycles, e.g. $a_1$ and $a_2$; we have $\beta=a_1+a_2+a_3$ and $\alpha$ passes from $p_1$ to $p_2$, we may choose it to cross $a_3$. So, in the basis $\{\alpha,\beta,a_1,a_2\}$, we may write $a_3=\beta-a_1-a_2$. 

The effect of $\omega'$ is now to rotate these triangles by one step; we can see that $a_1\mapsto a_2$, $a_2\mapsto a_3=\beta-a_1-a_2$, and $\beta\mapsto \beta$, the only nontrivial thing is $\omega'(\alpha)$. The cycle $\omega'(\alpha)$ passes from $p_1$ to $p_2$ crossing $a_1$ rather than $a_3$, so by adding $\frac13\beta$ near each of the $p_i$ we can form the cycle $\omega'(\alpha)-\alpha+\frac23\beta$, which is homologous to $a_1+a_3$. In other words, we have $\omega'(\alpha)-\alpha+\frac23\beta = a_1+a_3= \beta-a_2$, or
\[
\omega'(\alpha) = \alpha + \frac13\beta -a_2
\]
and we are done.
\end{proof}

Now combining Propositions \ref{Prop:degree-contr} and \ref{Prop:1311-local-monodromy} gives us a main result of this paper:

\begin{Thm}\label{Thm:global-monodromy-1311}
  Let $P$ be a Fano polygon with singularity content $(k,\{n\times \frac13(1,1)\}$, and let $X_t$ be defined using a generic maximally mutable Laurent polynomial $f$ with $Newt(f)=P$.
Then there is a basis $\{\alpha,\beta,a^1_1,a^1_2,\ldots, a^n_1,a^n_2\}$ of cycles in $H_1(X_t,\ZZ)$ such that in terms of this basis, the monodromy automorphism $\omega$ of $H_1(X_t,\ZZ)$ is given by
  \begin{itemize}
  \item $\omega(\alpha) = \alpha + (k+2n-12)\beta -\sum_{j=1}^na_2^j$,
\item $\omega(\beta) = \beta$,
\item $\omega(a_1^j) = a_2^j$ for $1\le j\le n$, and
\item $\omega(a_2^j) = \beta-a_1^j-a_2^j$ for $1\le j\le n$.
  \end{itemize}
\end{Thm}

\begin{proof}
 For each edge $E$ with a $\frac13(1,1)$-cone, \ref{Prop:1311-local-monodromy} gives us two candidates for the local monodromy automorphism of $H_1(X_E)$: the automorphism $\omega'$, and its inverse; the only thing we have to do is find out which one gives integral global monodromy. Now,
\[
\omega'(\alpha) = \alpha + \frac13\beta -a_2
\]
while
\[
\omega'^{-1}(\alpha) = \alpha - \frac13\beta +a_1,
\]
and from \ref{Prop:degree-formula} (from which we compute $A(\frac13(1,1))=\frac53$) and \ref{Prop:degree-contr} we have that only $\omega'$ gives integral global monodromy.
\end{proof}

\begin{figure}
\centering
\begin{tikzpicture}[auto,>=stealth]
\node (00) at (0,0) {$\bullet$};
\node (24) at (2,4) {$\bullet$};
\node (40) at (4,0) {$\bullet$};
\node (46) at (6,4) {$\bullet$};
\node (hole1) at (2,1.5) {$p_1$};
\node (hole2) at (4,2.5) {$p_2$};
\draw[->] (00) to node [swap] {$a_2$} (40);
\draw[->] (24) to node [swap] {$a_1$} (00);
\draw[->,dashed] (40) to (24);
\draw[->] (24) to node {$a_2$} (46);
\draw[->] (46) to node {$a_1$} (40);
\node (a3label) at (3,3) {$a_3$};
\draw[->] (hole1) to (hole2);
\node (alphalabel) at (2.7,1.5) {$\alpha$};
\node (edge1) at (1,2) {};
\node (edge2) at (5,2) {};
\draw[->] (hole1) to (edge1);
\draw[->] (edge2) to (hole2);
\node (omegalabel1) at (1.3,1.4) {$\omega(\alpha)$};
\node (omegalabel2) at (4.7,2.6) {$\omega(\alpha)$};
\draw[->] (2.4,1.7) arc (35:135:0.5);
\draw[<-] (3.6,2.3) arc (215:315:0.5);
\node (betalabel1) at (2,2.3) {$\frac13\beta$};
\node (betalabel2) at (4,1.7) {$\frac13\beta$};
\end{tikzpicture}
\caption{The automorphism $\omega'$ from Proposition \ref{Prop:1311-local-monodromy}}\label{fig:1311-automorphism}
\end{figure}
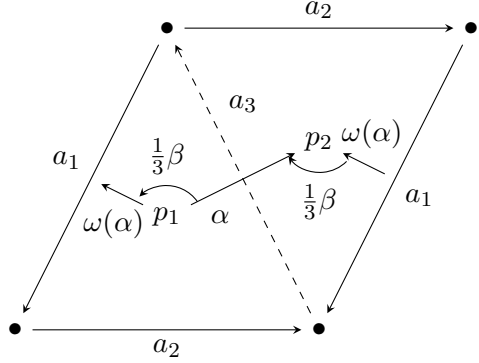

\subsubsection{The case of $g\ge 2$}

For general genus, we instead have a quotient of the hyperbolic plane, by some hyperbolic lattice. Here we are unable to pin down the fixpoints as we did in the previous case, but we can confine the possibilities very strongly (the following result is surely well known, but we are unable to locate an exact reference):

\begin{Lemma}\label{Lemma:hyperbolic-plane-automorphisms}
Up to conjugation, there are only $n$ order $n$ elements of $PSL_2(\RR)$.
\end{Lemma}

\begin{proof}
Elements of $PSL_2(\RR)$ are represented by $2\times 2$ real matrices with determinant one. A matrix of order $n$ must have as eigenvalues $n$'th roots of unity, and as the determinant is one, the two eigenvalues must be mutually inverse, and the trace must be $2\cos(\frac{2\pi}n)$. A matrix with this trace and determinant is of the form
\[
\begin{pmatrix}
  a & b\\
\frac1b(-a^2+2a\cos(2\pi/n)-1) & \cos(2\pi/n)-a\\
\end{pmatrix}
\]
where $b\neq 0$. All such matrices are similar, as the matrix with parameters $a,b$ can be transformed to the matrix with parameters $c,d$ by conjugating with the real matrix 
\[
\begin{pmatrix}
\frac{d (a^2-2\cos(2\pi/n) a+1)}{b(c^2-2\cos(2\pi/n) c+1)} & \frac{(a-c) d}{c^2-2\cos(2\pi/n) c+1} \\
 0 & 1 \\
\end{pmatrix}
\]
(this is ok as $x^2-2\cos(2\pi/n)x+1$ has no real roots). There are $n$ possible $n$'th roots of unity, so from this the result follows.
\end{proof}

The automorphism group of a genus $g\ge 2$ surface is a subgroup of $PSL_2(\RR)$, so this implies that if we can find one primitive order $h$ automorphism of our surface with the required number of fixpoints, its powers will give us all the other possibilities, up to choice of basis.

To give a complete description valid for any $P$ would require computing the number $A(\sigma)$ of Proposition \ref{Prop:degree-formula} for all singularity types $\sigma=\frac1r(1,a)$ corresponding to an $R$-singularity (it is always 1 for a primitive $T$-singularity), and finding an automorphism of a curve of the appropriate genus (with appropriate order and number of fixed points) that compensates for the non-integer part of $A(\sigma)$. We will do this by giving a model for the surface in similar fashion to Figure \ref{fig:1311-automorphism} and finding a description in terms of cycles of the resulting automorphism. We can not give a concise general formula for what power of this automorphism is the right one, but it reduces case-by-case to modular arithmetic and is easily doable by hand; we compute some simple cases in Proposition \ref{Prop:powers}.

Let us first restrict attention to the simplest case of an $R$-cone with lattice width $1$; here by necessity the lattice height $r\ge 3$ must be odd (otherwise the cone can't be spanned by primitive lattice vectors).

Suppose now we have an edge $E$ with an $R$-cone of height $r$ and width one, and as in the genus one case let $X_E$ be the inverse image under $X_t\to X_0$ of the proper transform of $D_E$ under the resolution $\widetilde{Y_P}\to Y_P$. By \ref{Lemma:R-cone-monodromy} the local monodromy automorphism of $H_1(X_E)$ is induced by an order $r$ automorphism of a genus $(r-1)/2$ Riemann surface with three fixed points (two punctures, and one other distinguished point), and by \ref{Lemma:hyperbolic-plane-automorphisms} it is enough to find one such automorphism; one of its powers will be the one that induces the local monodromy automorphism on $H_1(X_E)$.

\begin{Prop}\label{Prop:w1-local-monodromy}
  Let $X$ be a Riemann surface of genus $(r-1)/2$, with two removed points $p_1,p_2$ and a third distinguished point $q$, and an order $r$ automorphism that fixes these points. Let $\alpha$ be a relative cycle passing from $p_1$ to $p_2$, and let $\beta$ be a cycle going once around $p_1$. Then there is a choice of cycles $c_1,\ldots,c_{r-1}$ such that $\{\alpha,\beta,c_1,\ldots,c_{r-1}\}$ is a basis for $H_1(X)$, and the automorphism is a power of the automorphism $\omega_r$ given on homology cycles by 
  \begin{itemize}
  \item $\alpha \mapsto \alpha +(1-\frac4r)\beta+c_1-\sum_{i=3}^{r-1}c_i$
\item $\beta \mapsto \beta$,
\item $c_i \mapsto c_{i+1}$ for $1\le i\le r-2$, and
\item $c_{r-1} \mapsto \beta - \sum_{i=1}^{r-1}c_i$.
  \end{itemize}
\end{Prop}

\begin{proof}
  Recall from the genus one case where we constructed a model for the surface by gluing together two triangles, and an order 3 automorphism with three fixpoints by rotating the triangles. For the $w=1,r\ge 3$ odd, case we will construct a model in a similar way; gluing together two $r$-gons to make a genus $(r-1)/2$ surface, and getting an order $r$ automorphism with three fixpoints by rotating the $r$-gons. 

More precisely: Take two regular $r$-gons, with edges labelled $c_1,\ldots,c_r$ going around in the positive direction. On one $r$-gon choose an orientation for each edge, and one the other give the corresponding edges the opposite orientation. Finally, identify edges with the same label according to their orientation; this gluing is easily seen to give a surface of the desired genus. Let the punctures $p_i$ be the centroids of the $r$-gons, and let $q$ be a vertex of an $r$-gon (these all are identified by the gluing). Let $\alpha$ be the relative cycle going from $p_1$ to $p_2$ across $c_r$, and let $\beta$ be the cycle going once around $p_1$. Choosing $2g=r-1$ of the $c_i$'s, e.g. $c_1,\ldots,c_{r-1}$, we have a basis $\{\alpha,\beta,c_1,\ldots,c_{r-1}\}$ of $H_1(X)$; the cycle $c_r$ can be expressed as $\beta-\sum_{i=1}^{r-1}c_i$ (see Figure \ref{fig:omega-5} for an illustration).

Now define the automorphism $\omega_r$ by 
\[
c_1\mapsto c_2\mapsto \cdots \mapsto c_r\mapsto c_1,
\]
which we may visualize as rotating the $r$-gons in the positive direction. It is clear that $\omega_r$ has order $r$ and fixes $p_1,p_2$ and $q$. The cycle $\beta$, which can be identified with $\sum_{i=1}^rc_r$, is clearly fixed. For the cycle $\alpha$, its image $\omega_r(\alpha)$ goes from $p_1$ to $p_2$ crossing $c_1$, and similar to the genus one case, the cycle $\omega_r(\alpha)-\alpha+\frac2r\beta$ is homologous to $c_1+c_r=\beta-\sum_{i=2}^{r-1}c_i$. From this we find
\[
\omega_r(\alpha)=\alpha + \frac{r-2}{r}\beta -\sum_{i=2}^{r-1}c_i.
\]
\end{proof}

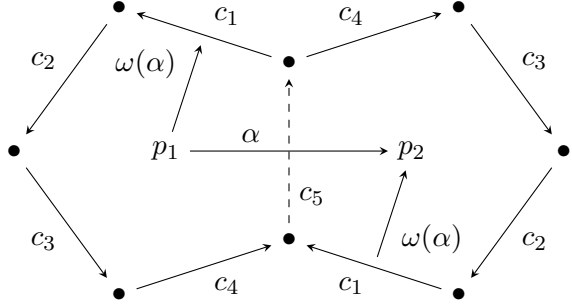
\begin{figure}
\centering
\begin{tikzpicture}[auto,>=stealth]
\node (p1) at (0,0) {$p_1$};
\node (q1) at (36:2) {$\bullet$};
\node (q2) at (108:2cm) {$\bullet$};
\node (q3) at (180:2cm) {$\bullet$};
\node (q4) at (-108:2cm) {$\bullet$};
\node (q5) at (-36:2cm) {$\bullet$};
\node (p2) at (4*cos 36,0) {$p_2$};
\node (q8) at (4*cos 36 + 2*cos 72 ,2*sin 72 ) {$\bullet$};
\node (q6) at (4*cos 36 + 2*cos 72,-2*sin 72) {$\bullet$};
\node (q7) at (2+4*cos 36,0) {$\bullet$};
\draw [->] (q1) to node [swap] {$c_1$} (q2);
\draw [->] (q2) to node [swap] {$c_2$} (q3); 
\draw [->] (q3) to node [swap] {$c_3$} (q4); 
\draw [->] (q4) to node [swap] {$c_4$} (q5); 
\draw [->,dashed] (q5) to (q1);
\draw [->] (q1) to node {$c_4$} (q8); 
\draw [->] (q8) to node {$c_3$} (q7); 
\draw [->] (q7) to node {$c_2$} (q6); 
\draw [->] (q6) to node {$c_1$} (q5);
\draw [->] (p1) to (p2);

\node (half1) at (cos 36+cos 108,sin 36+sin 108) {};
\draw [->] (p1) to node {$\omega(\alpha)$} (half1);

\node (half2) at (3*cos 36 - cos 108,-sin 36-sin 108) {};
\draw [->] (half2) to node [swap] {$\omega(\alpha)$} (p2);

\node (c5label) at (2*cos 36+.3,-.6) {$c_5$};
\node (alphalabel) at (2*cos 36-.5,.2) {$\alpha$};

\end{tikzpicture}
\caption{The automorphism $\omega_5$ in the model of a genus 2 surface.}\label{fig:omega-5}
\end{figure}

Now as above let $E$ be an edge of $P$ with an $R$-cone of height $r$ and width one. By \ref{Lemma:R-cone-monodromy}, \ref{Lemma:hyperbolic-plane-automorphisms} and \ref{Prop:w1-local-monodromy} the local monodromy automorphism of $H_1(X_E)$ is induced by some power of the map $\omega_r$ defined in \ref{Prop:w1-local-monodromy}. To find which it is necessary to compute the number $A(\sigma)$ for the singularity type $\sigma=\frac1r(1,a)$ of the $R$-cone, and find a power $\omega_r^k$ of $\omega_r$ such that $\omega_r^k(\alpha)+A(\sigma)\beta$ is integral in the basis from \ref{Prop:w1-local-monodromy}. We will do this explicitly for the simplest cases $\frac1r(1,a)$ with $a=1,2$ or 3; it is not a hard computation, but we do not have a general expression yet, and must do it case-by-case.

It is straightforward to verify that
\[
\omega_r^k(\alpha)=\alpha+(1-\frac{2k}r)\beta+\sum_{i=1}^{k-1}c_i-\sum_{i=k+1}^{r-1}c_i,
\]
so for each $\sigma=\frac1r(1,a)$, we want to find a $k$ such that $A(\sigma)+(1-\frac{2k}r)$ is an integer. 

\begin{Prop}\label{Prop:powers}
  Let $A(\sigma)$ be the number defined in \ref{Prop:degree-formula}, and let $m(\sigma)$ be an integer such that $A(\sigma)+1-\frac{2m(\sigma)}r$ is an integer. Then
  \begin{enumerate}
  \item $m(\frac1r(1,1)) \equiv_r -2$,
\item $m(\frac1r(1,2)) \equiv_r 2-l^2$, where $r=2l+1$, and
\item $m(\frac1r(1,3)) \equiv_r   
\begin{cases}    
2(l-1) &\text{ if }r=3l+1\\
-8(l+1) &\text{ if }r=3l+2.  
\end{cases}$
  \end{enumerate}

\end{Prop}

\begin{proof}
Recall that $r$ is odd in all cases. We compute $A(\frac1r(1,1))=6-r-\frac4r$, so $A(\sigma)+1-\frac{2m}r\equiv_\ZZ0$ reduces to $-4-2m\equiv_r0$, which as $r$ is odd gives $m(\frac1r(1,1))\equiv_r-2$.

If now $r=2l+1$, we have that $A(\frac1r(1,2))=5-l+\frac{l-4}r$; we now solve the congruence $l-4-2m\equiv_r0$ and find the solution $m\equiv_r2-l^2$.

For the $\frac1r(1,3)$ case we must separate into the cases $r=3l+1$ (here $l$ is even) and $r=3l+2$ (here $l$ is odd). In the former case, $A(\frac1r(1,3))=6-l+\frac1r(l-5)$, and the conguence $l-5-2m\equiv_r0$ has the solution $m\equiv_r2(l-1)$. For the latter case, we have $A(\frac1r(1,3))=5-l-(l+6)/r$, and the congruence $-l-6-2m\equiv_r0$ has the solution $m\equiv_r-8(l+1)$.
\end{proof}

For the general $g\ge 2$ case, with width $w\ge 2$, we can do essentially the same thing as above; in the general case there are however several choices to be made during the process, and no apparent means of identifying the ``right one''. The ambiguity in choices can be isolated to the image of the relative cycle $\alpha$. Recall that above we had $\omega(\alpha)=\alpha+\frac{h-2}h\beta+\sum \pm c_i$, where the $c_i$'s were basis elements constructed from the edges of the polygons we glued to obtain our surface; in the general case we will have a similar formula, but several possible choices for what should go in the sum. As the most interesting part of $\omega(\alpha)$ is the coefficient of $\beta$, this can be considered a minor loss.

\begin{Prop}\label{Prop:genl-case}
  Let $w\ge 2$, and let $X$ be a Riemann surface of genus $\frac12(h-1)w$ with two removed points $p_1,p_2$ and $w$ distinguished points $q_1,\ldots,q_w$, and an order $h$ automorphism fixing these points. Let $\alpha$ be a relative cycle passing from $p_1$ to $p_2$, and let $\beta$ be a cycle going once around $p_1$. Then there is a choice of cycles $e_1,\ldots,e_{(h-1)w}$ such that $\{\alpha,\beta,e_1,\ldots,e_{(h-1)w}\}$ is a basis for $H_1(X)$, and the automorphism is a power of the map $\omega$ given on homology cycles by
  \begin{itemize}
\item $\omega(\beta)=\beta$,
  \item $\omega(e_j)=
\begin{cases}
  e_{j+w} & j\le(h-2)w\\
\beta-\sum_{k\neq 1}e_{j+kw} & j>(h-2)w,
\end{cases}$ and
\item $\omega(\alpha)=\alpha+(1-\frac2h)\beta-\sum_{k=w}^{h-1}e_{kw}+(\ast)$,
  \end{itemize}
here $(\ast)$ is a $\ZZ$-linear combination of $e_j$'s, and the indices in the second item must be taken modulo $(h-1)w$.
\end{Prop}

\begin{proof}
  As in Propositions \ref{Prop:w1-local-monodromy} and \ref{Prop:1311-local-monodromy}, we make a model for $X$ by gluing two polygons as follows: Take two regular $hw$-gons, and for the first label the centroid by $p_1$ and the vertices in anticlockwise order by $q_1,q_2,\ldots,q_w,q_1,q_2,\ldots q_w$, and label the edges by $c_1,\ldots c_{hw}$, oriented anticlockwise. On the second $hw$-gon, label the vertices in clockwise order $q_1,q_2,\ldots,q_w,q_1,q_2,\ldots q_w$, and label the edges with labels $c_1,\ldots c_{hw}$ in such a way that for each $i$, $c_{i+w}$ is $w$ steps from $c_i$ going anticlockwise, each $c_i$ is connecting the same points $q_j,q_{j+1}$ as in the first polygon, and such that $c_{i+1}$ does not follow $c_i$ going clockwise. There are a total of $(h-1)^{w-2}(h-2)$ ways of doing this: at an edge connecting $q_j$ and $q_{j+1}$, there are $h$ possible labels $c_j,c_{j+w},\ldots,c_{j+hw}$ (reading the indices modulo $hw$), and there are $w$ such sets; within each such set the order is fixed by the requirement that $c_{i+w}$ and $c_i$ are separated by $w$ steps. As $c_{i+1}$ cannot follow $c_i$, each choice rules out a choice on the adjacent edges, so there are $h-1$ choices on each edge once one choice has been made, except the final one where there are only $h-2$ as it gets constraints imposed from both sides. By rotation of the polygon, we can regard one set $c_j,c_{j+w},\ldots,c_{j+hw}$ as fixed; this gives $(h-1)^{(w-2)}(h-2)$ possible arrangements. 

The demands that $c_i$ not be followed by $c_{i+1}$ and that each $c_i$ connects the same points $q_j,q_{j+1}$ in both polygons ensure that when gluing according to the labels, we get a surface of genus $\frac12(h-1)w$. The demand that $c_i$ is separated by $w$ steps from $c_{i+w}$ preserves the order $h$ automorphism given by ``rotation by $\frac1h$'': $c_i\mapsto c_{i+w}$; this is our $\omega$.

Now define the homology basis by $e_j:=\sum_{k=0}^{w-1}c_{j+k}$, $\beta$ is equivalent to $\sum_{i=0}^{hw}c_i$, and we may take $\alpha$ to pass from $p_1$ to $p_2$ crossing $c_{hw}$. Applying $c_{hw}=\beta-\sum_{i=0}^{hw-1}c_i$, it is easy to check that
\[
\omega(e_j)=
\begin{cases}
  e_{j+w} & j\le(h-2)w\\
\beta-\sum_{k\neq 1}e_{j+kw} & j>(h-2)w,
\end{cases}
\]
and it is obvious that $\omega(\beta)=\beta$. For the relative cycle $\alpha$, as in \ref{Prop:w1-local-monodromy} we see that $\omega(\alpha)-\alpha+\frac2h\beta$ is homologous to $c_w+c_{hw}+\sum_{i=1}^{w-1}c_i+\sum_{i=1}^{w-1}c_{[i]}$, here $c_{[i]}=c_{i+kw}$ for some $k$ (this is where the ambiguity in labelling edges comes in, we can only fix these mod $w$). Rearranging using $\sum_{i=w}^{hw-1}c_i = \sum_{k=1}^{h-1}e_{kw}$, we have 
\[
\omega(\alpha)=\alpha+(1-\frac2h)\beta+\sum_{i=1}^{w-1}c_{[i]}+c_w-\sum_{k=1}^{h-1}e_{kw},
\]
and for each choice of $c_{[i]}$'s, we can of course express $\sum_{i=1}^{w-1}c_{[i]}+c_w$ in terms of the $e_j$'s, but there is no concise general formula.
\end{proof}

\begin{Ex}\label{Ex:1611-labels}
The only choice of $h,w$ with $w\ge 2$ that gives an unambiguous labelling is $h=3,w=2$, the $\frac16(1,1)$ $R$-cone. A picture of the labelling is given in Figure \ref{fig:1611-labels}.

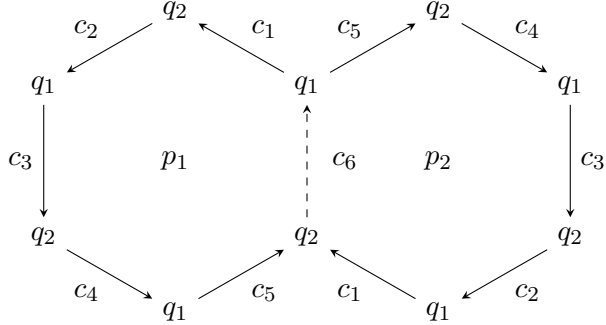
\begin{figure}
\centering
\begin{tikzpicture}[auto,>=stealth]
\node (p1) at (0,0) {$p_1$};
\node (q1) at (30:2) {$q_1$};
\node (q2) at (90:2cm) {$q_2$};
\node (q3) at (150:2cm) {$q_1$};
\node (q4) at (-150:2cm) {$q_2$};
\node (q5) at (-90:2cm) {$q_1$};
\node (q6) at (-30:2cm) {$q_2$};
\node (p2) at (4*cos 30,0) {$p_2$};
\node (q10) at (4*cos 30,2) {$q_2$};
\node (q9) at (4*cos 30+2*cos 30,2*sin 30) {$q_1$};
\node (q8) at (4*cos 30+2*cos 30,-2*sin 30) {$q_2$};
\node (q7) at (4*cos 30,-2) {$q_1$};
\draw [->] (q1) to node [swap] {$c_1$} (q2);
\draw [->] (q2) to node [swap] {$c_2$} (q3); 
\draw [->] (q3) to node [swap] {$c_3$} (q4); 
\draw [->] (q4) to node [swap] {$c_4$} (q5); 
\draw [->] (q5) to node [swap] {$c_5$} (q6); 
\draw [->,dashed] (q6) to (q1);
\draw [->] (q1) to node {$c_5$} (q10); 
\draw [->] (q10) to node {$c_4$} (q9); 
\draw [->] (q9) to node {$c_3$} (q8); 
\draw [->] (q8) to node {$c_2$} (q7);
\draw [->] (q7) to node {$c_1$} (q6);

\node (c6label) at (2*cos 30 +0.5,0) {$c_6$};

\end{tikzpicture}
\caption{Labels for the local model for $X_E$, for $h=3,w=2$.}\label{fig:1611-labels}
\end{figure}

\end{Ex}

\begin{Ex}\label{Ex:h3w3labels}
  For $h=w=3$ (the primitive $T$-cone of height 3), there are $(3-1)^{(3-2)}(3-2)=2$ possible labellings. If the first $9$-gon is labelled $c_1,\ldots,c_9$ (in anticlockwise order), the second must be labelled (in anticlockwise order) either $c_2,c_7,c_3,c_5,c_1,c_6,c_8,c_4,c_9$, or $c_5,c_1,c_3,c_8,c_4,c_6,c_2,c_7,c_9$.
\end{Ex}

\begin{Thm}\label{Thm:monodromy-sing-content}
  Let $P$ be a Fano polygon, let $f$ be a maximally mutable Laurent polynomial with $Newt(f)=P$, and let $L_f$ be the associated Picard-Fuchs operator. The monodromy at zero of $L_f$ determines and is determined by the singularity content of $P$ (thought of as a multiset).
\end{Thm}

\begin{proof}
  It is clear by \ref{Prop:degree-contr}, \ref{Prop:degree-formula}, \ref{Prop:1311-local-monodromy}, \ref{Prop:w1-local-monodromy} and \ref{Prop:genl-case} that the singularity content determines the monodromy.

Suppose now that the singularity content of $P$ is $(k,\mathcal{B})$, and that we are given the monodromy matrix in the bases we have described. By \ref{Prop:1311-local-monodromy}, \ref{Prop:w1-local-monodromy} and \ref{Prop:genl-case} this matrix is of the form 
\[
\begin{pmatrix}
  1 & 0& 0\\
B & 1& \mathbf{r}\\
\mathbf{c} & 0 & 
\begin{matrix}
  \mathbf{M_1} &             &        &\\
              & \mathbf{M_2} &        &\\
              &              &\ddots &\\
 &                         &        &\mathbf{M_n}\\
\end{matrix}
\\
\end{pmatrix}
\]
here $\mathbf{r}$ and $\mathbf{c}$ are some vectors and $\mathbf{M_i}$ are block matrices of size $2g_i\times 2g_i$ where $g_i$ is the genus of the local piece $X_{E_i}$, and $B=12-k-\sum_{\sigma\in\mathcal{B}}m(\sigma)$, where $m(\sigma)=A(\sigma)-(1-\frac{2p}h)$ and $p$ is the power of the local monodromy map required to make $A(\sigma)-(1-\frac{2p}h)$ an integer (as in Proposition \ref{Prop:powers}).

Each block $\mathbf{M_i}$ is associated to an $R$-cone of type $\frac1{r_i}(1,a_i)$. The sizes $2g_i$ of the blocks $\mathbf{M_i}$ give us the $r$ in $\frac1r(1,a)$, as $2g_i=w_i(h_i-1)$ ($h_i$ and $w_i$ are the height and width of the $R$-cone, respectively) and necessarily the matrix $\mathbf{M_i}$ has order $h_i$, so we can solve for $w_i$ and get $r_i=h_iw_i$; the $a_i$ can be deduced by finding the correct power $p_i$ (as done in \ref{Prop:powers}) for each $R$-cone of this height and width (the list is finite) and selecting the one that equals $\mathbf{M_i}$.

Now having identified the singularity basket $\mathcal{B}$, we deduce the number of $T$-cones as $k=12-B-\sum_{\sigma\in\mathcal{B}}m(\sigma)$.

That we cannot recover the cyclical order of the singularity basket follows from the easily verified fact that the local monodromy automorphisms over the $f$-rigid cones commute; also we may reorder the blocks $\mathbf{M_i}$ as desired by reordering the basis.
\end{proof}

\begin{Ex}
  The numbers $m(\sigma)$ for the cases $\frac1r(1,a)$ for $a=1,2,3$ are 
\begin{itemize}
\item $m(\frac1r(1,1)) = 5-r$,
\item $m(\frac1r(1,2)) = 0$, and
\item $m(\frac1r(1,3)) = \begin{cases}0 & \text{ if }r\equiv_3 1\\ -\frac23(r+1) & \text{ if }r\equiv_32\end{cases}$,
\end{itemize}
as can be easily computed from \ref{Prop:w1-local-monodromy} and \ref{Prop:powers}.
\end{Ex}

\begin{Cor}\label{Cor:eigenvalues-of-monodromy}
  With the assumptions of Theorem \ref{Thm:monodromy-sing-content}, suppose the singularity basket contains $n_i$ $R$-cones of height $h_i$. Then the monodromy of $L_f$ at zero has eigenvalues 1 with multiplicity 2, and each $h_i$'th root of unity (other than 1) with multiplicity $n_i$.
\end{Cor}

%
%

\section{Ramification and degree of $L_f$}\label{sec:ramification}

We now have a good description of the monodromy of $L_f$ around the origin. From this we can deduce some information about $L_f$, for instance we already know from \ref{Cor:order-of-L_f} that the order of $L_f$ (i.e. the highest degree in the differential variable $\nabla_t$) is twice the mutational genus. It is more difficult to prove anything about the degree (i.e. the degree in the variable $t$ of the leading term of $L_f$). We do have some observations and conjectures, however.

A local system $V$ on $\PP^1\setminus S$ (where $S$ is a finite set) has monodromy $T_s$ around each point $s\in S$, and we can gather up some information about the total monodromy group in a quantity called the \emph{ramification index} of $V$, defined by
\[
rf(V) = \sum_{s\in S}dim(V_x/V_x^{T_s}) - 2rk(V),
\]
where $x\in \PP^1\setminus S$ is some point (it doesn't matter which, as $T_s$ is only defined up to conjugation, i.e. up to choice of base point). It is a general fact that $rf(V)\ge 0$, in particular local systems with $rf(V)=0$ seem interesting in their own right (also see \cite{CCGGK}). 

The ramification index measures the sizes of the eigenspaces associated to the eigenvalue 1 at each singular point. With $V=Sol(L_f)$ at the singular point $t=0$, we see that eigenspace has dimension either one or two, depending on whether the number $B$ defined in the proof of Theorem \ref{Thm:monodromy-sing-content} is one or zero, respectively. Both cases occur, for instance there are 26 mutation classes of polygons with singularity content $(k,\{n\times \frac13(1,1)\})$ (see \cite{Minimal-polygons}), and of these 6 have $B=0$ and the rest have $B>0$.

The origin thus contributes either $dim(Sol(L_f))-1$ or $dim(Sol(L_f))-2$ to the ramification. We can re-express the ramification defect as
\[
rf(Sol(L_f))=\sum_{s\in S\setminus\{0\}}dim(Sol(L_f)_x/Sol(L_f)_x^{T_s})-2g_{mut}(Y_P)-\delta
\]
where $S$ is the singular locus of $L_f$, $\delta$ is 1 or 2, and $rk(Sol(L_f))=2g_{mut}$. From now on, we write $rf(L_f)$ for $rf(Sol(L_f))$ for simplicity. Now let $E_i$ be the dimension of the eigenspace of 1 at the singular point $s_i\in S\setminus\{0\}$, then using $|S\setminus \{0\}|=deg(L_f)$, we have after some rearrangement that (writing $d=deg(L_f)$ and $g=g_{mut}$)
\[
rf(Sol(L_f))=2g(d-1)-\delta-\sum E_i.
\]
We have some empirical evidence of some further information: for those instances of $L_f$ that have been explicitly computed, which are the smooth Fano polygons and several of those with singularity content $(k,\{ n\times\frac13(1,1)\})$, a pattern emerges:

\begin{Conj}\label{Conj:degree-and-rf-conjectures}
Let $P$ be a Fano polygon with singularity content $(k,\{ n\times\frac13(1,1)\})$, let $f$ be a maximally mutable Laurent polynomial with $Newt(f)=P$, and let $L_f$ be the associated Picard-Fuchs operator. Then
\begin{enumerate}
\item the degree of $L_f$ is equal to $g^2+3g-1+2g\cdot rf(L_f)$, and
\item the ramification index $rf(L_f)$ is equal to $n+k_{eff}-3$, where $k_{eff}$ is the number of multiple points on the curve $f=0$.
\end{enumerate}
\end{Conj}

\begin{Rem}The number $k_{eff}$ here is equal to $k$ for the generic MMLP's, and drops by one whenever two $T$-cones on the same edge of $P$ have the same associated factor $(\gamma+\eta x)$ in $f$. Thus, the minimal possible value for $rf(L_f)$ occurs when all the $T$-cones have the same factor (e.g. in the standard MMLP case), and is equal to the minimal number of vertices of polygons mutation-equivalent to $P$, minus three. We should point out that \ref{Conj:degree-and-rf-conjectures} only applies to polygons with singularity basket $\{ n\times\frac13(1,1)\}$; it is not entirely clear how to generalize it. As an example, the polygon with vertices $(-1,0),(2,1),(3,-1)$ has singularity content $(2,\{ 1\times\frac15(1,1)\})$; the conjecture would predict a degree of 17 for the standard MMLP, but the actual value is 19.
\end{Rem}

\begin{Ex}\label{Ex:ord-deg-rf}The computations are expensive, as noted at the start of Section \ref{sec:prelims}, and the output is very large and not particularly enlightening, so we'll show only the simplest few examples here. The simplest smooth Fano polygon polygon is the one with vertices $(0,1)$, $(1,0)$ and $(-1,-1)$, with singularity content $(3,\varnothing)$: the standard MMLP is $x+y+\frac1{xy}$ and $L_f$ is $\nabla^2-27t^3(\nabla+1)(\nabla+2)$ (as before, $\nabla=t\dd_t$); this has ramification index zero, degree 3 and order 2. The second simplest smooth Fano polygon is the one with vertices $(0,1)$, $(1,0)$, $(-1,-1)$ and $(1,1)$, with singularity content $(4,\varnothing)$; here the standard MMLP is $x+y+\frac1{xy}+xy$ and $L_f$ is $8\nabla^2+t\nabla(17\nabla-1)-t^2(5\nabla+8)(11\nabla+8)-12t^3(30\nabla^2+78\nabla+47)-4t^4(\nabla+1)(103\nabla+147)-99t^5(\nabla+1)(\nabla+2)$, this has ramification index 1, order 2 and degree 5. We may observe (though we don't know how to prove this) that there is no way to mutate this polygon into one with three vertices. 

There are nonequivalent polygons with the same singularity content, but giving different ramification index for the associated $L_f$'s. The simplest example is for singularity content $(5,\{1\times \frac13(1,1)\}$: for the polygon with vertices $(-3,1)$, $(3,1)$ and $(0,-1)$, we have order 4, degree 9, and ramification index zero; for the polygon with vertices $(-1,-1)$, $(-1,2)$, $(1,1)$ and $(2,-1)$ we have order 4, degree 13, and ramification index one.
\end{Ex}

\begin{Rem}
  Conjecture \ref{Conj:degree-and-rf-conjectures} suggests a way to distinguish nonequivalent mutation classes of Fano polygons with the same singularity content; in the minimal case the value for $k_{eff}$ is the minimal number of edges (or vertices) of polygons mutation-equivalent to $P$. This number together with the singularity content could be an invariant that completely classifies Fano polygons up to mutation.
\end{Rem}

Conjecture \ref{Conj:degree-and-rf-conjectures}, if true, gives us the ramification index and degree directly from $Newt(f)$, and so gives some bounds on the $E_i$. We have
\[
\sum E_i=2g(d-1)-\delta-rf(L_f),
\]
and as there are $d$ singular points, we can write each $E_i=2g-\eps_i$, where $\sum_i\eps_i = 2g+\delta+rf(L_f)$. This number is always smaller than the degree (assuming \ref{Conj:degree-and-rf-conjectures}), so there are guaranteed to be at least $d-\sum_i\eps_i=g^2+g-2+(2g-\delta)rf(L_f)$ points with trivial monodromy and possibly more. This is quite special, as a generic local system has nontrivial monodromy at every singular point.

\section*{Acknowledgements}

This paper came out of the project to classify Fano manifolds via mirror symmetry, to which I was introduced during the PRAGMATIC 2013 Research school in Algebraic Geometry and Commutative Algebra “Topics in Higher Dimensional Algebraic Geometry” held in Catania, Italy, in September 2013. I am very grateful to Alfio Ragusa, Francesco Russo, and Giuseppe Zappalá, the organizers of the PRAGMATIC school, for making all that happen. I am also grateful to Alessio Corti and Al Kasprzyk for introducing me to this topic and for the fruitful collaboration that has followed, and to Tom Coates and Rikard Bøgvad for helpful comments.


\begin{thebibliography}{BGK{\etalchar{+}}87}

\bibitem[ACC{\etalchar{+}}15]{overarching}
Mohammad Akhtar, Tom Coates, Alessio Corti, Liana Heuberger, Alexander
  Kasprzyk, Alessandro Oneto, Andrea Petracci, Thomas Prince, and Ketil
  Tveiten, \emph{Mirror symmetry and the classification of orbifold del {P}ezzo
  surfaces}, arXiv:1501.05334 (2015).

\bibitem[ACGK12]{ACGK}
Mohammad Akhtar, Tom Coates, Sergey Galkin, and Alexander~M. Kasprzyk,
  \emph{Minkowski polynomials and mutations}, SIGMA Symmetry Integrability
  Geom. Methods Appl. \textbf{8} (2012), Paper 094, 17.

\bibitem[AK14]{Al-Mo-singularity-content}
Mohammad Akhtar and Alexander Kasprzyk, \emph{Singularity content},
  arXiv:1401.5458v1 (2014).

\bibitem[BGK{\etalchar{+}}87]{Borel}
A.~Borel, P.-P. Grivel, B.~Kaup, A.~Haefliger, B.~Malgrange, and F.~Ehlers,
  \emph{Algebraic {$D$}-modules}, Perspectives in Mathematics, vol.~2, Academic
  Press Inc., Boston, MA, 1987.

\bibitem[CCG{\etalchar{+}}13]{CCGGK}
Tom Coates, Alessio Corti, Sergey Galkin, Vasily Golyshev, and Alexander
  Kasprzyk, \emph{Mirror symmetry and fano manifolds}, European Congress of
  Mathematics (Krak{\'o}w, 2-7 July, 2012), November 2013, pp.~285--300.

\bibitem[CLS11]{Cox-Little-Schenk}
David~A. Cox, John~B. Little, and Henry~K. Schenck, \emph{Toric varieties},
  Graduate Studies in Mathematics, vol. 124, American Mathematical Society,
  Providence, RI, 2011.

\bibitem[GH94]{Griffiths-Harris}
Phillip Griffiths and Joseph Harris, \emph{Principles of algebraic geometry},
  Wiley Classics Library, John Wiley \& Sons, Inc., New York, 1994, Reprint of
  the 1978 original.

\bibitem[H{\"o}r90]{Horm}
Lars H{\"o}rmander, \emph{The analysis of linear partial differential
  operators. {I}}, second ed., Grundlehren der Mathematischen Wissenschaften
  [Fundamental Principles of Mathematical Sciences], vol. 256, Springer-Verlag,
  Berlin, 1990, Distribution theory and Fourier analysis.

\bibitem[KN12]{KN12}
Alexander~M. Kasprzyk and Benjamin Nill, \emph{Fano polytopes}, Strings, Gauge
  Fields, and the Geometry Behind -- the Legacy of {M}aximilian {K}reuzer
  (Anton Rebhan, Ludmil Katzarkov, Johanna Knapp, Radoslav Rashkov, and Emanuel
  Scheidegger, eds.), World Scientific, 2012, pp.~349--364.

\bibitem[KNP15]{Minimal-polygons}
Alexander~M. Kasprzyk, Benjamin Nill, and Thomas Prince, \emph{Minimality and
  mutation-equivalence of polygons}, arXiv:1501.05335 (2015).

\bibitem[KT15]{MaxMutMedAl}
Alexander~M. Kasprzyk and Ketil Tveiten, \emph{Maximally mutable laurent
  polynomials}, In preparation (2015).

\bibitem[Lai14]{Lairez}
Pierre Lairez, \emph{Computing periods of rational integrals},
  arXiv:1404.5069v2 (2014).

\bibitem[OP15]{Ale-Andrea}
Alessandro Oneto and Andrea Petracci, \emph{Quantum periods of del pezzo
  surfaces with $\frac13(1,1)$ singularities}, In preparation (2015).

\bibitem[SST00]{SST}
Mutsumi Saito, Bernd Sturmfels, and Nobuki Takayama, \emph{Gr\"obner
  deformations of hypergeometric differential equations}, Algorithms and
  Computation in Mathematics, vol.~6, Springer-Verlag, Berlin, 2000.

\bibitem[{\.Z}o{\l}06]{Zoladek}
Henryk {\.Z}o{\l}adek, \emph{The monodromy group}, Instytut Matematyczny
  Polskiej Akademii Nauk. Monografie Matematyczne (New Series) [Mathematics
  Institute of the Polish Academy of Sciences. Mathematical Monographs (New
  Series)], vol.~67, Birkh\"auser Verlag, Basel, 2006.

\end{thebibliography}

\newcommand{\etalchar}[1]{$^{#1}$}
\providecommand{\bysame}{\leavevmode\hbox to3em{\hrulefill}\thinspace}
\providecommand{\MR}{\relax\ifhmode\unskip\space\fi MR }
\providecommand{\MRhref}[2]{%
  \href{http://www.ams.org/mathscinet-getitem?mr=#1}{#2}
}
\providecommand{\href}[2]{#2}

\end{document}